\documentclass[preprint]{elsarticle}

\usepackage{tikz}
\usetikzlibrary{intersections,calc,arrows.meta}

\usepackage[hang,small,bf]{caption}
\usepackage[subrefformat=parens]{subcaption}
\usepackage{ascmac}

\usepackage{pdfpages} 

\usepackage{bm}
\usepackage{amsmath}
\usepackage{mathrsfs}
\usepackage{amssymb}
\usepackage{booktabs}
\usepackage{comment}
\usepackage{mathtools}
\usepackage{multirow}
\usepackage{graphicx}
\usepackage{pgfplots}
\pgfplotsset{compat=1.18}
\numberwithin{equation}{section}

\usepackage{amsthm}

\newtheorem{thm}{Theorem}[section]
\newtheorem{lem}[thm]{Lemma}
\newdefinition{rem}{Remark}[section]

\newtheorem{mainthm}{Main Theorem}
\newdefinition{dfn}{Definition}[section]
\newtheorem{prop}{Proposition}[section]
\newtheorem{cor}{Corollary}[section]
\newdefinition{ex}{Example}[section]
\newdefinition{ass}{Assumption}

\begin{document}

\begin{frontmatter}


  \title{Learning Coefficients in Semi-Regular Models}
  \author{Yuki Kurumadani} 
  \ead{kurumadani@sigmath.es.osaka-u.ac.jp}
  \affiliation{organization={Graduate School of Engineering Science, Osaka University},
            addressline={1 Chome-3 Machikaneyamacho}, 
            city={Toyonaka},
            postcode={Osaka 560-0043}, 
            country={Japan}}
  \begin{abstract}
Recent advances have clarified theoretical learning accuracy in Bayesian inference, revealing that the asymptotic behavior of metrics such as generalization loss and free energy, assessing predictive accuracy, is dictated by a rational number unique to each statistical model, termed the learning coefficient (real log canonical threshold) \cite{watanabe1}. For models meeting regularity conditions, their learning coefficients are known \cite{watanabe1}. However, for singular models not meeting these conditions, exact values of learning coefficients are provided for specific models like reduced-rank regression \cite{Aoyagi1}, but a broadly applicable calculation method for these learning coefficients in singular models remains elusive.

The problem of determining learning coefficients relates to finding normal crossings of Kullback-Leibler divergence in algebraic geometry \cite{watanabe1}. In this context, it is crucial to perform appropriate coordinate transformations and blow-ups. 

This paper introduces an approach that utilizes properties of the log-likelihood ratio function for constructing specific variable transformations and blow-ups to uniformly calculate the real log canonical threshold. It was found that linear independence in the differential structure of the log-likelihood ratio function significantly influences the real log canonical threshold. This approach has not been considered in previous research.

In this approach, the paper presents cases and methods for calculating the exact values of learning coefficients in statistical models that satisfy a simple condition next to the regularity conditions (\textbf{semi-regular models}), offering examples of learning coefficients for two-parameter semi-regular models and mixture distribution models with a constant mixing ratio.

  \end{abstract}
  
\end{frontmatter}

\section{Introduction}
\subsection{Definitions and Assumptions}
Throughout this paper, we consider statistical models \(p(x|\theta)\) with continuous parameters \(\theta = (\theta_1, \ldots, \theta_d) \in \Theta (\subset \mathbb{R}^d) (d\geq1)\), and we denote the true distribution by \(q(x)\). 
The set of possible values for data \(x\) is denoted by \(\chi\). 
The statistical model is defined as \textbf{realizable}, meaning there exists a parameter \(\theta_*\) such that \(q(x) = p(x|\theta_*)\)a.s.. Such parameters \(\theta_*\) are referred to as realization parameters, and the entire set of these parameters is denoted by \(\Theta_*\). 
\textbf{The prior distribution \(\varphi(\theta)\) is assumed to satisfy \(\varphi(\theta_*) > 0\) for any realization parameter \(\theta_*\).}
The random variable that follows the true distribution \(q\) is denoted by \(X\), and \(\mathbb{E}_X[\cdot]\) denotes the operation of taking the average concerning the random variable \(X\). 
In this paper, we assume that the operations of taking expectations and partial derivatives with respect to \(\theta\) are interchangeable.

The Kullback-Leibler divergence is defined as:
\[
K(\theta) := \mathbb{E}_X\left[\log {\frac{p(X|\theta_*)}{p(X|\theta)}}\right]
\]
and is assumed to be analytic around \(\theta = \theta_*\). The log-likelihood ratio function is given by:
\[
f(x|\theta) := \log {\frac{p(x|\theta_*)}{p(x|\theta)}}
\]
and is assumed to be \(L^2\) integrable and analytic around \(\theta = \theta_*\).

For fixed data \(x\), the \(m\)-th order terms of the Taylor expansion of the log-likelihood ratio function \(f(x|\theta)\) at \(\theta = \theta_*\) for \(\theta_1, \ldots, \theta_s\) (\(s \leq d\)) are defined as:
\[
F_m(x|\theta_1,\ldots,\theta_s) := \sum_{\substack{i_1+\cdots+i_s=m\\ i_1,\ldots,i_s \in \mathbb{Z}_{\geq 0}}}\frac{1}{i_1!\cdots i_s!}
\cdot\left.\frac{\partial^m f(x|\theta)}{\partial \theta_1^{i_1}\cdots\partial \theta_s^{i_s}}\right|_{\theta = \theta_*}
\cdot(\theta_1 - \theta_{1*})^{i_1}\cdots(\theta_s - \theta_{s*})^{i_s}
\]

The Fisher information matrix at \(\theta = \theta_*\) is denoted by \(I\), i.e.,
\[
I = \mathrm{Cov}\left(\left.\frac{\partial \log{p(X|\theta)}}{\partial \theta_i}\right|_{\theta = \theta_*} \cdot \left.\frac{\partial  \log{p(X|\theta)}}{\partial \theta_j}\right|_{\theta = \theta_*}\right)_{i,j=1,\ldots,d}
\]
and its rank is denoted by \(r\). The Hessian matrix of \(K(\theta)\) at \(\theta = \theta_*\) is denoted by \(J\), i.e.,
\[
J = \left(\left.\frac{\partial^2 K(\theta)}{\partial \theta_i \partial \theta_j}\right|_{\theta = \theta_*}\right)_{i,j=1,\ldots,d}
\]
Generally, \(I\) and \(J\) do not coincide, but under the conditions of this paper, as will be seen later in Remark~\ref{remIJ}, they do coincide.

\subsection{What is a Learning Coefficient?}
First, let us describe the framework of Bayesian theory. Consider a statistical model \(p(x|\theta)\) with parameters \(\theta \in \mathbb{R}^d (d\geq1)\), where \(x \in \chi\), and the true distribution is denoted by \(q(x)\). 
The posterior distribution of the parameter \(\theta\) given data \(X_1, \ldots, X_n\) independently drawn from \(q(x)\) is expressed as:
\[
p(\theta|X_1,\ldots,X_n) = \frac{\varphi(\theta) \prod_{i=1}^n p(X_i|\theta)}{Z_n}
\]
where \(Z_n := \int \varphi(\theta) \prod_{i=1}^n p(X_i|\theta) d\theta\) represents the marginal likelihood function. 
When a new observation \(x\) is given, a statistical model marginalized over its posterior distribution is called a predictive distribution, represented by:
\[
p^*(x) := \int p(x|\theta) p(\theta|X_1,\ldots,X_n) d\theta
\]
This predictive distribution is used to estimate the true distribution \(q(x)\) under Bayesian theory.

The accuracy of the predictive distribution is measured by metrics such as the generalization loss \(G_n\) and the free energy \(F_n\):
\begin{align*}
G_n &:=-\int_{\chi} q(x) \log{p^*(x)} dx\\
F_n &:=-\log{Z_n} = -\log{\int \varphi(\theta) \prod_{i=1}^n p(X_i|\theta) d\theta}
\end{align*}
These metrics are known to take smaller values when the predictive distribution closely approximates the true distribution.

When the true distribution $q(x)$ is realizable by the statistical model $p(x|\theta)$, the generalization loss $G_n$ and the free energy $F_n$ exhibit the following asymptotic behavior (i.e., the behavior as the sample size $n$ becomes large) using a positive rational number $\lambda$ and an integer $m$ greater than or equal to 1 \cite{watanabe1}.
\begin{align*}
\mathbb{E}[G_n] &= -\int_{\chi} q(x) \log{q(x)} dx + \frac{\lambda}{n} - \frac{m-1}{n \log{n}} + o\left(\frac{1}{n \log{n}}\right)\\
\mathbb{E}[F_n] &= -n \int_{\chi} q(x) \log{q(x)} dx + \lambda \log{n} - (m-1) \log{\log{n}} + O(1)
\end{align*}

Since the first term on the right-hand side of both expressions does not depend on the statistical model \(p(x|\theta)\), the asymptotic behaviors are determined by \(\lambda\) in the second term. Thus, when comparing two learning models using these metrics, the value of \(\lambda\) becomes crucial to determine which model better approximates the true distribution. Generally, this \(\lambda\) is referred to as the \textbf{learning coefficient}. The learning coefficient \(\lambda\) is defined for the trio of the statistical model \(p(x|\theta)\), the true distribution \(q(x)\), and the prior distribution \(\varphi(\theta)\).

\subsection{Methods for Calculating Learning Coefficients Using Algebraic Geometry}
It is known that the concept of learning coefficients coincides with the algebraic geometry concept of the \textbf{real log canonical threshold}. The real log canonical threshold is defined using a technique known as resolution of singularities. Here, resolution of singularities refers to transforming an analytic function \(F\) into normal crossing as specified by the following theorem:

\begin{thm}[Resolution of Singularities]\label{resolution thm}\cite{hironaka1}\cite[Theorem 2.3]{watanabe1}
Let \(F(x)\) be a real analytic function defined near the origin in \(\mathbb{R}^d\) and assume \(F(0) = 0\). Then, there exists an open set \(W \subset \mathbb{R}^d\) containing the origin, a real analytic manifold \(U\), and a proper analytic map \(g: U \rightarrow W\) satisfying the following conditions:
\begin{itemize}
    \item[(1)] Define \(W_0 := F^{-1}(0)\) and \(U_0 := g^{-1}(W_0)\). The map \(g: U-U_0 \rightarrow W-W_0\) is an analytic isomorphism.
    \item[(2)] At any point \(Q\) in \(U_0\), by taking local coordinates \(u = (u_1, \ldots, u_d)\) with \(Q\) as the origin, we can express:
    \begin{align}\label{eq:normalcrossing}
    F(g(u)) = a(u) u_1^{k_1} u_2^{k_2} \cdots u_d^{k_d}\\
    \left|g^{\prime}(u)\right| = \left|b(u) u_1^{h_1} u_2^{h_2} \cdots u_d^{h_d}\right|\notag
    \end{align}
    where \(k_i, h_i (i=1, \ldots, d)\) are non-negative integers, and \(a(u), b(u)\) are real analytic functions defined near the origin in \(\mathbb{R}^d\) with \(a(0) \neq 0, b(0) \neq 0\).
\end{itemize}
\end{thm}

The expression as in (\ref{eq:normalcrossing}) is referred to as \textbf{normal crossing}.

It should be noted that Theorem~\ref{resolution thm} is a local statement concerning the neighborhood of the origin in $\mathbb{R}^d$. 
In other words, if there are multiple points where singularities need to be resolved, it is necessary to find $(W,U,g)$ as guaranteed by this Theorem~\ref{resolution thm} at each point.

\begin{dfn}[Real Log Canonical Threshold]\label{LambdaDef}
Let $F$ be a real analytic function defined on an open set $O$ in $\mathbb{R}^d$, and let $C$ be a compact set containing $O$. For each point $P$ in $C$ satisfying $F(P)=0$, by applying a coordinate transformation such that point $P$ is moved to the origin of $\mathbb{R}^d$, Theorem~\ref{resolution thm} can be applied. Fix $(W,U,g)$ as guaranteed by Theorem~\ref{resolution thm}(2). Additionally, denote the non-negative integers $h_i,k_i$ given by Theorem~\ref{resolution thm}(2) in the neighborhood of any point $Q\in U_0$ as $h_i^{(Q)}, k_i^{(Q)}$.
\begin{itemize}
\item[(1)]
  Define the \textbf{real log canonical threshold} $\lambda_P$ at point $P$ of the function $F$ as:
\[
\lambda_P=\inf_{Q\in U_0}\left\{\min_{i=1,\ldots,d}{\frac{h_i^{(Q)}+1}{k_i^{(Q)}}}\right\}
\]
where $(h_i+1)/k_i=\infty$ if $k_i=0$. It is known that this is well-defined, i.e., it does not depend on the choice of $(W,U,g)$.\cite[Theorem 2.4]{watanabe1}
\item[(2)]
  Define the \textbf{real log canonical threshold} $\lambda$ for the compact set $C$ of the function $F$ as\cite[Definition 2.7]{watanabe1}:
  \[
  \lambda = \inf_{P\in C}\lambda_P
  \]
  \item[(3)]
  In (2), for the point $P (\in C)$ that gives the minimum value, the maximum number of $i$ such that $\lambda_P=(h_i^{(Q)}+1)/k_i^{(Q)}$ is satisfied is called the \textbf{multiplicity}. (If there are multiple points $P (\in C)$ that give the minimum value, the maximum number of $i$ among each is called the multiplicity.)
\end{itemize}

\end{dfn}

In this paper, we apply Theorem~\ref{resolution thm} to the analytic function \(F\) defined as the Kullback-Leibler divergence \(K(\theta)\). As previously seen, we assumed a prior distribution $\varphi(\theta_*)>0$ in this paper. Under this condition, the \textbf{learning coefficient \(\lambda\) is equal to the real log canonical threshold for the compact set \(\Theta_* = \{\theta \in \Theta | K(\theta) = 0\}\)} \cite[Theorem 6.6, Definition 6.4]{watanabe1}. That is, by performing the resolution of singularities guaranteed by Theorem~\ref{resolution thm} at each point $P$ of $\Theta_*$, we obtain the real log canonical threshold $\lambda_P$. The minimum value of $\lambda_P$ as point $P$ moves over the entire $\Theta_*$ coincides with the learning coefficient $\lambda$. Moreover, if the real log canonical threshold $\lambda_P$ at some points $P$ on $\Theta_*$ can be calculated, it is clear from the definition that this provides an upper bound for the learning coefficient. That is, $\lambda\leq\lambda_P$ holds.

It should be noted that, although $K(P_1)=K(P_2)=0$ holds for any elements $P_1, P_2 \in \Theta_*$, the differential structure of the function $K$ in the neighborhoods of these two points may not be identical, and thus the real log canonical thresholds $\lambda_{P_1}, \lambda_{P_2}$ may not coincide. Therefore, to obtain the learning coefficient, it is necessary to calculate the real log canonical threshold $\lambda_P$ for all points $P$ in $\Theta_*$.

\subsection{Learning Coefficients for Models Satisfying Regularity Conditions}

It is known that for learning models satisfying the regularity conditions, which allow the posterior distribution to converge to a normal distribution, the learning coefficient \(\lambda\) is given by \(\lambda = d/2\), where \(d\) is the dimension of the parameter space \cite{watanabe1}. Here, the regularity conditions are defined as follows:
\begin{dfn}[Regularity]\label{def1}
A true model \(q(x)\) is said to be \textbf{regular} with respect to the statistical model \(p(x|\theta)\) if it satisfies the following three conditions:
\begin{itemize}
    \item[(1)] There is only one element in the set of realization parameters \(\Theta_*\).
    \item[(2)] \(\theta_*\) is an interior point of the parameter set \(\Theta\), meaning there exists an open neighborhood \(\tilde{\Theta} \subset \Theta\) around \(\theta_*\).
    \item[(3)] The Fisher information matrix \(I \in \mathbb{R}^{d \times d}\) is a positive definite matrix.
\end{itemize}
\end{dfn}

This paper aims to generalize this formula. More specifically, it considers the case where the rank \(r\) of the Fisher information matrix \(I\) is \(0 < r < d\) to provide the learning coefficient \(\lambda\).

\subsection{Upper Bound of Learning Coefficients for General Models}

Assuming a prior distribution \(\varphi(\theta_*) > 0\) throughout this paper, it is known that the real log canonical threshold \(\lambda\) of \(K(\theta)\) at \(\theta = \theta_*\) satisfies\cite[Theorem 7.2]{watanabe1}:
\begin{equation}\label{eq:zyokai01}
\lambda \leq \frac{d}{2}
\end{equation}

\newpage
\section{Main Theorem}\label{sec:mainthm}
\subsection{Overview of the Main Theorem}

\begin{dfn}[Semi-Regularity]\label{def3}
A statistical model \(p(x|\theta)\) is said to be \textbf{semi-regular\footnote{This terminology is not a general term.}} at the parameter \(\theta_*\) that realizes the true model \(q(x)\) if the rank \(r\) of the Fisher information matrix \(I\in \mathbb{R}^{d\times d}\) at \(\theta=\theta_*\) is greater than zero.
\end{dfn}

Clearly, a regular model is a semi-regular model. In the following, we arbitrarily fix an element \(\theta_*\) in \(\Theta_*\) and consider the real log canonical threshold at this point. We assume (by translating if necessary) that \(\theta_* = 0\). In the subsequent discussion, we will denote it as \(\theta_*\) for statements that hold regardless of whether \(\theta_* = 0\), but it can be read as \(\theta_* = 0\) without loss of generality.

\begin{ass}\label{mainass}
In the Main Theorem of this paper, we assume the following (1)-(3) for semi-regular models.
\begin{itemize}
  \item[(1)] For the $r$ parameters \(\theta_{1},\ldots,\theta_{r}\), the $r$ random variables
  \[
\left.\frac{\partial f(X|\theta)}{\partial\theta_{1}}\right|_{(\theta_1,\ldots,\theta_d)=0},\ldots,\left.\frac{\partial f(X|\theta)}{\partial\theta_{r}}\right|_{(\theta_1,\ldots,\theta_d)=0}
  \]
  are linearly independent\footnote{In this paper, random variables are said to be linearly independent if they are so with probability 1.}.
  \item[(2)] For the remaining $d-r$ parameters $\theta_{r+1}, \ldots, \theta_{d}$, let \(m\) be an integer greater than or equal to 1. The derivatives of the log-likelihood ratio function \(f(X|\theta)\) with respect to \(\theta_{r+1},\ldots,\theta_{d}\) up to order \(m-1\) at \((\theta_1,\ldots,\theta_d)=0\) are zero with probability 1. In other words,
  \[
  F_m(X|\theta_1,\ldots,\theta_s)
  :=\sum_{\substack{i_1+\cdots+i_s=m\\ i_1,\ldots,i_s\in\mathbb{Z}_{\geq 0}}}
  \frac{1}{i_1!\cdots i_{s}!}
  \cdot
  \left.\frac{\partial^m f(X|\theta)}{\partial\theta_1^{i_1}\cdots\partial\theta_{s}^{i_{s}}}\right|_{(\theta_1,\ldots,\theta_d)=0}
  \cdot\theta_1^{i_1}\cdots\theta_s^{i_s}
  \]
is used such that
\[
F_1(X|\theta_{r+1}, \ldots, \theta_d) = \cdots = F_{m-1}(X|\theta_{r+1}, \ldots, \theta_d) = 0\ \ \text{a.s.}
\]

The maximum value among such \(m\) is redefined as \(m\). 
For convenience, if \(m=1\), it is treated as the case where \(r=d\) in (1).

\item[(3)] For each of the \(d-r\) parameters \((\theta_{r+1},\ldots,\theta_{d}) \neq 0\), one of the following holds:
  \begin{itemize}
    \item[(i)] \(F_m(X|\theta_{r+1},\ldots,\theta_{d})=0\ \ (\text{a.s.})\)
    \item[(ii)] \(F_m(X|\theta_{r+1},\ldots,\theta_{d})\) and the $r$ random variables in (1) are linearly independent.
  \end{itemize}
  For \(m=1\), it is always treated as satisfying (ii) for convenience.
\end{itemize}
\end{ass}

\begin{rem}
In Section~\ref{sec:chg}, we discuss the method for constructing coordinate transformations that satisfy Assumption~\ref{mainass}.
\end{rem}

For semi-regular models that satisfy Assumption~\ref{mainass}(1)(2), the Taylor expansion of the log-likelihood ratio function \(f(X|\theta)\) in terms of the parameters \(\theta\) does not contain terms of order less than \(m-1\) for \((\theta_{r+1}, \ldots, \theta_d)\), implying:
\[
f(X|\theta) = F_1(X|\theta_1, \ldots, \theta_{r}) + F_m(X|\theta_{r+1}, \ldots, \theta_{d}) + \text{(higher order terms)}\ \text{a.s.}
\]
Taking the expected value with respect to \(X\) yields the Taylor expansion of \(K(\theta)\) at \(\theta = 0\):
\[
K(\theta) = \mathbb{E}_X\left[F_1(X|\theta_1, \ldots, \theta_{r}) + F_m(X|\theta_{r+1}, \ldots, \theta_{d})\right] + \text{(higher order terms)}
\]
However, it turns out that many of the lower order terms vanish upon taking the expectation. This formulation is presented in Main Theorem~\ref{mainthm}.

\begin{mainthm}\label{mainthm}
For semi-regular models that satisfy Assumption~\ref{mainass}(1)(2), the Taylor expansion of \(K(\theta)\) at \(\theta=0\) can be expressed as:
\[
K(\theta)=\frac{1}{2}\mathbb{E}_X\left[\left\{F_1(X|\theta_1,\ldots,\theta_{r})+F_m(X|\theta_{r+1},\ldots,\theta_{d})\right\}^2\right]+\text{(higher order terms)}
\]
The (higher order terms) do not include:
\begin{itemize}
  \item Terms up to the \(2m\)-th order that consist only of \(\theta_{r+1},\ldots,\theta_{d}\)
  \item Terms that are first order in \(\theta_1,\ldots,\theta_r\) and up to \(m\)-th order in \(\theta_{r+1},\ldots,\theta_{d}\)
  \item Second-order terms that consist only of \(\theta_1,\ldots,\theta_r\)
\end{itemize}
\begin{figure}[ht]
\centering
\begin{tikzpicture}

\coordinate (O) at (0,-0.1);
\coordinate (A) at (0,0.1);
\coordinate (B) at (1.8,0.1);
\coordinate (C) at (1.8,-0.1);
\filldraw[fill=blue!60!white][ultra thin] (O)--(A)--(B)--(C)--(O);

\coordinate (O1) at (0,0.9);
\coordinate (A1) at (0,1.1);
\coordinate (B1) at (1.8,1.1);
\coordinate (C1) at (1.8,0.9);
\filldraw[fill=red!60!white][ultra thin] (O1)--(A1)--(B1)--(C1)--(O1);

\coordinate (O2) at (2,-0.1);
\coordinate (A2) at (2,0.1);
\coordinate (B2) at (3.8,0.1);
\coordinate (C2) at (3.8,-0.1);
\filldraw[fill=red!60!white][ultra thin] (O2)--(A2)--(B2)--(C2)--(O2);

\draw[->,>=stealth,semithick](-0.5,0)--(6+0.5,0)node[below right]{$\theta_{r+1},\ldots,\theta_{d}$ degree};
\draw[->,>=stealth,semithick](0,-1)--(0,2.5+0.5)node[left]{$\theta_{1},\ldots,\theta_{r}$ degree};
\draw(0,0)node[below left]{O};
\draw[dotted](0,1)--(2,1);
\draw[dotted](1.8,0)--(1.8,1);
\draw[dotted](2,0)--(2,1);

\draw(2,-0.1)node[below]{$m$};
\draw(4,0)node[below right]{$2m$};
\draw(0,1)node[left]{$1$};
\draw(0,2)node[left]{$2$};
\draw(1.5,-0.5)node[below]{$m-1$};
\draw[->,>=stealth,ultra thin](1.5,-0.5)--(1.8,0);
\draw(3.5,-0.5)node[below]{$2m-1$};
\draw[->,>=stealth,ultra thin](3.5,-0.5)--(3.8,0);

\fill (2,1) circle (2.5pt);
\fill (4,0) circle (2.5pt);
\fill (0,2) circle (2.5pt);
\fill (1.8,0) circle (1.5pt);
\fill (3.8,0) circle (1.5pt);
\fill (0,1) circle (1.5pt);

\draw [arrows = {-Stealth[scale=1.5]}]  (3,1.5) to [out=180,in=60] (2,1);
\draw [arrows = {-Stealth[scale=1.5]}]  (5,0.5) to [out=180,in=60] (4,0);
\draw [arrows = {-Stealth[scale=1.5]}]  (1,2.5) to [out=180,in=60] (0,2);

\draw(3,1.5)node[right]{$\mathbb{E}_X\left[F_1(X|\theta_{1},\ldots,\theta_{r})F_m(X|\theta_{r+1},\ldots,\theta_{d})\right]$};
\draw(5,0.5)node[right]{$\frac{1}{2}\mathbb{E}_X\left[F_m^2(X|\theta_{r+1},\ldots,\theta_{d})\right]$};
\draw(1,2.5)node[right]{$\frac{1}{2}\mathbb{E}_X\left[F_1^2(X|\theta_{1},\ldots,\theta_{r})\right]$};

\end{tikzpicture}
\caption{Main Theorem~\ref{mainthm}. Assuming that the derivatives of the log-likelihood ratio function are zero (a.s.) in the blue area (including the endpoints), the derivatives of the log-likelihood ratio function in the red area (including the endpoints) become random variables with expected value zero. This implies that the coefficients of \(K(\theta)\)'s Taylor expansion in these regions are all zero.}
\end{figure}
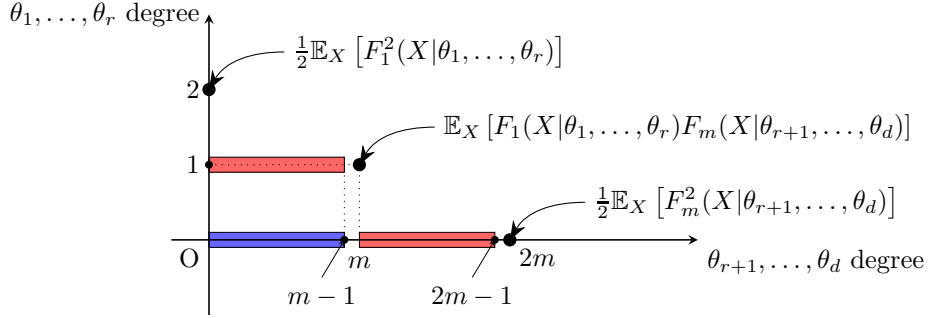
\end{mainthm}

\begin{mainthm}[Formula for the Real Log Canonical Threshold]\label{mainthm1}
Consider a semi-regular model that satisfies Assumption~\ref{mainass}(1)(2)(3). Consider the following blow-up \(g\) at the origin \(O\):
\begin{itemize}
  \item [(a)] Perform one blow-up centered at the origin of \(\mathbb{R}^d\).
  \item [(b)] If the exceptional surface from (a) is \(\{\theta_i=0\}\) (where \(i=r+1,\ldots,d\)), then further perform a blow-up centered at the subvariety
              \(\{(\theta_1,\ldots,\theta_d) \mid \theta_1=\cdots=\theta_r=\theta_i=0\}\).
  \item [(c)] If the exceptional surface from (b) is \(\{\theta_i=0\}\), repeat (b) until the total number of blow-ups reaches \(m\).
\end{itemize}

That is, at the $m$-th blow-up, when the exceptional surface is $\{\theta_i=0\}$, the map $g=g_i$ can be expressed as follows for $(i=r+1,\ldots,d)$:
\[
  g_i\colon(\theta^{\prime}_1,\ldots,\theta^{\prime}_{i-1},\theta_i,
  \theta^{\prime}_{i+1},\ldots,\theta^{\prime}_d)
  \mapsto
  (\theta_1,\ldots,\theta_{i-1},\theta_i,\theta_{i+1},\ldots,\theta_d);
\]
\begin{align*}
  &\theta_1=\theta^m_i\theta_1^{\prime},\ldots,\theta_r=\theta^m_i\theta_r^{\prime},\\
  &\theta_{r+1}=\theta_i\theta_{r+1}^{\prime},\ldots,\theta_{i-1}=\theta_i\theta_{i-1}^{\prime},\ 
  \theta_{i+1}=\theta_i\theta_{i+1}^{\prime},\ldots,\theta_{d}=\theta_i\theta_{d}^{\prime}
\end{align*}
  Defining the subset $S$ of $U_0:=g^{-1}(O)$ by local coordinates $(\theta^{\prime}_1,\ldots,\theta^{\prime}_{i-1},\theta_i,\theta^{\prime}_{i+1},\ldots,\theta^{\prime}_d)$ as
    \[
  S:=\bigcup_{i=r+1}^d
  \left\{
  (\theta^{\prime}_1,\ldots,\theta^{\prime}_d)~\middle|
  \begin{array}{l}
    (\theta^{\prime}_1,\ldots,\theta^{\prime}_{r},\theta_i)=0 \\
    F_m(X|\theta^{\prime}_{r+1},\ldots,\theta^{\prime}_{i-1},1,\theta^{\prime}_{i+1},\ldots,\theta^{\prime}_{d})= 0 (\text{a.s.})
  \end{array}
  \right\}\subset U_0\]  
  then, on \(U_0\setminus S\), normal crossing of \(K(\theta)\) is obtained, and
  \begin{equation}\label{eq_mainthm00}
    \inf_{Q\in U_0\setminus S}\left\{\min_{i=1,\ldots,d}{\frac{h_i^{(Q)}+1}{k_i^{(Q)}}}\right\}=\frac{d-r+rm}{2m}
  \end{equation}
  is satisfied (multiplicity is 1). See Definition~\ref{LambdaDef} for the symbols $k_i^{(Q)},h_i^{(Q)} $.

   Particularly, if all parameters $(\theta_{r+1}, \ldots, \theta_d) \neq 0$ satisfy Assumption~\ref{mainass}(3)$(\mathrm{ii})$, then the real log canonical threshold $\lambda_O$ at the origin $O$ of $K(\theta)$ is given by the following equation (with multiplicity 1):
\begin{equation}\label{eq_mainthm01}
\lambda_O = \frac{d-r + rm}{2m}
\end{equation}
Under this condition, the only point in the neighborhood of the origin in the parameter space $\Theta$ that satisfies $K(\theta) = 0$ is the origin itself.

\end{mainthm}

\begin{rem}
In the case of regular models (\(r=d\)), substituting \(r=d\) into the result (\ref{eq_mainthm01}) of Main Theorem~\ref{mainthm1} gives the real log canonical threshold:
\[
\lambda_O=\frac{d-d+d\times m}{2\times m}=\frac{d}{2}
\]
which indeed matches the results of prior research. 
Furthermore, considering that (\ref{eq_mainthm01}) monotonically decreases with respect to \(m\), it is also evident that the results of Main Theorem~\ref{mainthm1} (\ref{eq_mainthm01}) are consistent with the outcomes of prior research (\ref{eq:zyokai01}).

Using Main Theorem~\ref{mainthm1} suggests that one can discuss the real log canonical threshold from the linear independence (or dependence) of the random variables that appear as coefficients in the lower-order terms of the Taylor expansion of the log-likelihood ratio function, without specifically considering the Taylor expansion of \(K(\theta)\). In other words, there is \textbf{no need to take expectations with respect to the random variable \(X\)}.

\end{rem}

\begin{rem}
If \(2\lambda\) is considered as an indicator of the complexity of a statistical model, the result for regular models implies that \(2\lambda\) matches the number of parameters \(d\), meaning that in the case of singular models, \(2\lambda\) is less than or equal to \(d\) as per (\ref{eq:zyokai01}). From this perspective, considering the result of Main Theorem~\ref{mainthm1} (\ref{eq_mainthm01}) as follows
\[
\frac{d-r+rm}{2m}=\frac{1}{2}\left(r+\frac{d-r}{m}\right)
\]
suggests that for linearly independent parameters \(\theta_1,\ldots,\theta_r\), each counts as one, and for parameters \(\theta_{r+1},\ldots,\theta_d\), where derivatives up to \(m-1\) are all zero, each should be counted as \(1/m\).

\end{rem}

\begin{rem}
Consider the real log canonical threshold (\ref{eq_mainthm01}) of Main Theorem~\ref{mainthm1} from the perspective of ideals. For the parameter \((\theta_1,\ldots,\theta_d)\), consider the ideal
\[
I:=\left(\theta_1,\ldots,\theta_r,\theta_{r+1}^m+\cdots+\theta_{d}^m\right)
\]
The real log canonical threshold (\ref{eq_mainthm01}) of Main Theorem~\ref{mainthm1} is the same as the real log canonical threshold of the ideal \(I\), i.e., the polynomial
\[
\theta_1^2+\cdots +\theta_r^2+\left(\theta_{r+1}^m+\cdots+\theta_{d}^m\right)^2
\]
can resolve singularities by the same blow-ups.
\end{rem}

Before proving Main Theorem~\ref{mainthm},~\ref{mainthm1}, a practical example is introduced.

\begin{ex}[Case of $m=2$]\label{ex:m=2}
  Let \(X\) be a random variable following a binomial distribution Bin$(2,\theta)$ with parameter \(\theta\) where \(0<\theta<1\).
  \[
  \tilde{p}(X=x|\theta)=\binom{2}{x}\theta^x(1-\theta)^{2-x}=
  \begin{cases}
    (1-\theta)^2, & (x=0)\\
    2\theta(1-\theta), & (x=1)\\
    \theta^2, & (x=2)
  \end{cases}
  \]

  Consider a mixed distribution model with parameters \((\theta_1,\theta_2)\) given by:
  \begin{equation}\label{eq:exkongou01}
  p(X=x|\theta_1,\theta_2):=
  \frac{1}{2}\cdot\tilde{p}\left(X=x\middle|\theta_1-\theta_2+\frac{1}{2}\right)
  +\frac{1}{2}\cdot\tilde{p}\left(X=x\middle|\theta_2+\frac{1}{2}\right)~~ (x=0,1,2)
  \end{equation}
  
  Assuming the true distribution is \(\tilde{p}(X|1/2)\), note that this model realizes the true distribution at \((\theta_1,\theta_2)=0\).

  Let us verify that the Main Theorem~\ref{mainthm} holds at $(\theta_1,\theta_2)=0$.
  By setting $d=2, r=1$ and demonstrating that $F_1(X|\theta_1)$ and $F_2(X|\theta_2)$ are non-zero,
  it follows that Assumption~\ref{mainass}(1)(2) is satisfied in the case of $m=2$.

  Using a computer program mathematica, we obtain\footnote{See Appendix for calculations.}:
  \begin{gather}
    K(\theta)=\theta_1^2+8\theta_2^4+8\theta_1^2\theta_2^2-16\theta_1\theta_2^3+\cdots\label{ex_m=2_eq01}\\
    F_1(X|\theta_1)=\left.\frac{\partial f(X|\theta)}{\partial\theta_1}\right|_{(\theta_1,\theta_2)=0}\theta_1 =
    \begin{cases}
      2\theta_1, & x=0\\
      0, & x=1\\
      -2\theta_1, & x=2\\
    \end{cases}\label{ex_m=2_eq02}\\
    F_1(X|\theta_2)=\left.\frac{\partial f(X|\theta)}{\partial\theta_2}\right|_{(\theta_1,\theta_2)=0}\theta_2=0\\
    F_2(X|\theta_2)=\frac{1}{2}\left.\frac{\partial^2 f}{\partial\theta_2^2}\right|_{(\theta_1,\theta_2)=0}\theta_2^2=
    \begin{cases}
      -4\theta_2^2, & x=0\\
      4\theta_2^2, & x=1\\
      -4\theta_2^2, & x=2
    \end{cases}\label{ex_m=2_eq03}
  \end{gather}
  Terms of order five and higher are omitted in (\ref{ex_m=2_eq01}). It is verified with probability 1 that (\ref{ex_m=2_eq02}) and (\ref{ex_m=2_eq03}) are non-zero.

  The applicability of Main Theorem~\ref{mainthm} is verified using (\ref{ex_m=2_eq02}) and (\ref{ex_m=2_eq03}), yielding:
  \begin{align*}
    &\frac{1}{2}\mathbb{E}_X\left[\left\{F_1(X|\theta_1)+F_2(X|\theta_2)\right\}^2\right]\\
    =&\frac{1}{2}\left\{\left(2\theta_1-4\theta_2^2\right)^2\cdot\frac{1}{4}
      +\left(0+4\theta_2^2\right)\cdot\frac{1}{2}
      +\left(-2\theta_1-4\theta_2^2\right)^2\cdot\frac{1}{4}\right\}
    =\theta_1^2+8\theta_2^4
  \end{align*}
  which shows that \(\theta_1^2\) and \(8\theta_2^4\) are not part of the (higher order terms) in Main Theorem~\ref{mainthm}, and other terms are included in (higher order terms). Thus, as claimed by the Main Theorem~\ref{mainthm}, the (higher order terms) do not include:
  terms of degree four or less consisting only of $\theta_2$,
  terms of first degree in $\theta_1$ and second degree or less in $\theta_2$,
  terms of second degree consisting only of $\theta_1$.

  Next, verify that Main Theorem~\ref{mainthm1} holds. First, note that (\ref{ex_m=2_eq02}) and (\ref{ex_m=2_eq03}) are linearly independent for any \(\theta_2\neq 0\) when \(\theta_1=1\), so this statistical model satisfies Assumption~\ref{mainass}(3)($\mathrm{ii}$).

  Thus, applying Main Theorem~\ref{mainthm1} for \((d,r,m)=(2,1,2)\), the real log canonical threshold at the origin should be \(3/4\) (multiplicity is 1).
  This is verified by performing the blow-up \(g_1\) centered at \((\theta_1,\theta_2)=0\) and seeking the normal crossing of \(K(\theta)\).
  
  \begin{itemize}
    \item [(a)]
    First, transform (\ref{ex_m=2_eq01}) with \(\theta_2=\theta_1\theta_2^{\prime}\) using power series \(h_1,a_1\) to get,
    \[
    K(\theta)
    =\theta_1^2\left\{
    1
    +8\theta_1^2\theta_2^{\prime 4}
    +\theta_1h_1(\theta_1,\theta_2^{\prime})
    \right\}
    =\theta_1^2a_1(\theta_1,\theta_2^{\prime})
    \]
    Since on any point of $g_1^{-1}(0)=\{(\theta_1,\theta_2^{\prime})|\theta_1=0\}$,
    \[
    \forall \theta_2^{\prime},~ a_1(0,\theta_2^{\prime})=1\neq 0
    \]
    normal crossings are obtained in this local coordinate $(\theta_1,\theta_2^{\prime})$.

    \item[(b)]
    Next, transform (\ref{ex_m=2_eq01}) with \(\theta_1=\theta_2\theta_1^{\prime}\) using power series \(h_2,a_2\) to get,
    \[
    K(\theta)
    =\theta_2^2\left\{
    \theta_1^{\prime 2}
    +8\theta_2^{2}
    +\theta_2h_2(\theta_1^{\prime},\theta_2)
    \right\}
    =\theta_2^2a_2(\theta_1^{\prime},\theta_2)
    \]
    where \(g_1^{-1}(0)=\{(\theta_1^{\prime},\theta_2)|\theta_2=0\}\) on any point except at \((\theta_1^{\prime},\theta_2)=0\),
    \[
    \forall \theta_1^{\prime}\neq 0,~ a_2(\theta_1^{\prime},0)=\theta_1^{\prime 2}\neq 0
    \]
    shows a normal crossing is achieved.

    Therefore, it is sufficient to find the normal crossings at the point $(\theta_1^{\prime},\theta_2)=0$. Further blow-up \(g_2\) is performed centered at this point.

    \item[(b1)]
    First, transform using \(\theta_2=\theta_1^{\prime}\theta_2^{\prime\prime}\) with power series \(h_3,a_3\) to get,
    \begin{align*}
    K(\theta)
    =\theta_1^{\prime 4}\theta_2^{\prime\prime 2}
    \left\{
    1
    +8\theta_2^{\prime\prime 2}
    +\theta_1^{\prime}h_3(\theta_1^{\prime},\theta_2^{\prime\prime})
    \right\}
    =\theta_1^{\prime 4}\theta_2^{\prime\prime 2}
    a_3(\theta_1^{\prime},\theta_2^{\prime\prime})    
    \end{align*}
    where \(g_2^{-1}(0)=\{(\theta_1^{\prime},\theta_2^{\prime\prime})|\theta_1^{\prime}=0\}\) on any point,
    \[
    \forall \theta_2^{\prime\prime},~ a_3(0,\theta_2^{\prime\prime})
    =1
    +8\theta_2^{\prime\prime 2}
    \neq 0
    \]
    shows a normal crossing is achieved.

    \item[(b2)]
    Next, transform using \(\theta_1^{\prime}=\theta_2\theta_1^{\prime\prime}\) with power series \(h_4,a_4\) to get,
    \[
    K(\theta)
    =\theta_2^4\left\{
    \theta_1^{\prime\prime 2}
    +8
    +\theta_2h_4(\theta_1^{\prime\prime},\theta_2)
    \right\}
    =\theta_2^4a_4(\theta_1^{\prime\prime},\theta_2)
    \]
    where \(g_2^{-1}(0)=\{(\theta_1^{\prime\prime},\theta_2)|\theta_2=0\}\) on any point,
    \[
    \forall\theta_1^{\prime\prime},~ 
    a_4(\theta_1^{\prime\prime},0)=
    \theta_1^{\prime\prime 2}
    +8\neq 0
    \]
    shows a normal crossing is achieved.

\end{itemize}

  The normal crossings for each local coordinate are summarized as shown in Table \ref{fig_ex:m=2}(see Definition~\ref{LambdaDef} for notation \(k_i^{(Q)},h_i^{(Q)}\)). Thus, the real log canonical threshold is confirmed to be \(3/4\) (multiplicity is 1).
\begin{table}[htbp]
  \centering
  \caption{Normal Crossings for Each Local Coordinate.}
  \label{fig_ex:m=2}
  \begin{tabular}{ccccccc}
    No. & loc.coord. & $K(\theta)$ & Jacobian & $k_i^{(Q)}$ & $h_i^{(Q)}$ & $\lambda$ \\\hline
    (a) & $(\theta_1,\theta_2^{\prime})$ & 
 $\theta_1^2a_1(\theta_1,\theta_2^{\prime})$ & $\theta_1 $ & $(2,0)$ & $(1,0)$ & 1 \\ \hline
    (b) & $(\theta_1^{\prime},\theta_2)\neq 0$ & 
 $\theta_2^2a_2(\theta_1^{\prime},\theta_2)$ & $\theta_2$ & $(0,2)$ & $(0,1)$ & 1 \\ \hline
    (b1) & $(\theta_1^{\prime},\theta_2^{\prime\prime})$ & 
 $\theta_1^{\prime 4}\theta_2^{\prime\prime 2}a_3(\theta_1^{\prime},\theta_2^{\prime\prime})$ & $\theta_1^{\prime 2}\theta_2^{\prime\prime} $ & $(4,2)$ & $(2,1)$ & 3/4 \\ \hline
    (b2) & $(\theta_1^{\prime\prime},\theta_2)$ & 
 $\theta_2^4a_4(\theta_1^{\prime\prime},\theta_2)$ & $\theta_2^2 $ & $(0,4)$ & $(0,2)$ & 3/4 \\ 
  \end{tabular}
\end{table}
  
\end{ex}

\subsection{Proof of Main Theorem~\ref{mainthm}}
The proof of Main Theorem~\ref{mainthm} is divided into several steps. 
Proofs other than that of Main Theorem~\ref{mainthm} are included in the appendix.

Initially, we present formulas for higher derivatives of the log-likelihood ratio function
\[
f(x|\theta) = \log \frac{p(x|\theta_*)}{p(x|\theta)}
\]
and the Kullback-Leibler divergence $K(\theta)$
\[
K(\theta) = \mathbb{E}_X\left[f(X|\theta)\right].
\]
For this purpose, we define the quantity $G_{\theta_{i_1}\ldots\theta_{i_n}}(x,\theta)$ for $1 \leq i_1, \ldots, i_n \leq d$ and $x \in \chi, \theta \in \Theta$ as given by Equation~(\ref{DefG}).

Consider partitions of the set $\{1, \ldots, n\}$, where $n \geq 1$. For example,
for $n=2$, there is one possible partition: $\{\{1\},\{2\}\}$.
For $n=3$, there are four possible partitions: \\
$\{\{\{1\},\{2\},\{3\}\}, \{\{1\},\{2,3\}\}, \{\{2\},\{1,3\}\}, \{\{3\},\{1,2\}\}\}$

The set partitions are structured as:
\[
\bigcup_i U_i = \{1, \ldots, n\},\ U_i \cap U_j = \{\},\ U_i \neq \{\}, \{1, \ldots, n\}
\]
Each level of the set is arbitrarily ordered and presented as a sequence. For instance,
for $n=2$, it can be represented as $(T_1=((1),(2)))$.For $n=3$, it can be represented as $(T_1=((1),(2),(3)), T_2=((1),(2,3)), T_3=((2),(1,3)), T_4=((3),(1,2)))$
These sequences are then relabeled with $\{i_1, \ldots, i_n\}$ replacing $\{1, \ldots, n\}$, denoted as $S_{i_1, \ldots, i_n}$. For example:
\begin{align*}
S_{1,3} &= (T_1=((1),(3)))\\
S_{1,1,4} &= (T_1=((1),(1),(4)), T_2=((1),(1,4)), T_3=((1),(1,4)), T_4=((4),(1,1)))\\
S_{1,1,1} &= (T_1=((1),(1),(1)), T_2=((1),(1,1)), T_3=((1),(1,1)), T_4=((1),(1,1)))
\end{align*}
Here, we write:
\[
k \in U, U \in T, T \in S_{i_1, \ldots, i_n}
\]
to denote that the sequence $S_{i_1, \ldots, i_n}$ includes $T$, and $T$ includes $U$, which contains the integer $k$.
The length of $U$ is denoted by $|U|$, and we define:
\begin{equation}\label{DefG}
G_{\theta_{i_1} \ldots \theta_{i_n}}(x, \theta) := \sum_{T \in S_{i_1, \ldots, i_n}} \prod_{U \in T} \frac{\partial^{|U|}\log{p(x|\theta)}}{\prod_{k \in U} \partial \theta_k}
\end{equation}

For instance,
\begin{align*}
G_{\theta_1 \theta_3}(x, \theta) &= \frac{\partial \log{p(x|\theta)}}{\partial \theta_1} \frac{\partial \log{p(x|\theta)}}{\partial \theta_3}\\
G_{\theta_1 \theta_1 \theta_4}(x, \theta) &= \left\{\frac{\partial \log{p(x|\theta)}}{\partial \theta_1}\right\}^2 \frac{\partial \log{p(x|\theta)}}{\partial \theta_4} + 2 \cdot \frac{\partial \log{p(x|\theta)}}{\partial \theta_1} \frac{\partial^2 \log{p(x|\theta)}}{\partial \theta_1 \partial \theta_4}\\
&+ \frac{\partial \log{p(x|\theta)}}{\partial \theta_4} \frac{\partial^2 \log{p(x|\theta)}}{\partial \theta_1^2}\\
G_{\theta_1 \theta_1 \theta_1}(x, \theta) &= \left\{\frac{\partial \log{p(x|\theta)}}{\partial \theta_1}\right\}^3 + 3 \cdot \frac{\partial \log{p(x|\theta)}}{\partial \theta_1} \frac{\partial^2 \log{p(x|\theta)}}{\partial \theta_1^2}
\end{align*}

Clearly, the definition of $G_{\theta_{i_1} \ldots \theta_{i_n}}(x, \theta)$ does not depend on the definition of $S_{i_1, \ldots, i_n}$ (the method of ordering sets into sequences).
In the following, $G_{\theta_1 \theta_1 \theta_4}(x, \theta)$ and $G_{\theta_1 \theta_1 \theta_1}(x, \theta)$ are also denoted as $G_{\theta_1^2 \theta_4}(x, \theta)$ and $G_{\theta_1^3}(x, \theta)$, respectively.

\begin{lem}\label{lemG}
The function $G_{\theta_{i_1}\cdots\theta_{i_n}}(x,\theta)$ satisfies the recurrence relation
\begin{align*}
 &G_{\theta_{i_1}\cdots\theta_{i_{n+1}}}(x,\theta)\\
 =& G_{\theta_{i_1}\cdots\theta_{i_n}}(x,\theta)\cdot\frac{\partial\log{p(x|\theta)}}{\partial\theta_{i_{n+1}}}+\frac{\partial^n\log{p(x|\theta)}}{\partial\theta_{i_1}\cdots\partial\theta_{i_n}}\cdot\frac{\partial\log{p(x|\theta)}}{\partial\theta_{i_{n+1}}}+\frac{\partial  G_{\theta_{i_1}\cdots\theta_{i_n}}(x,\theta)}{\partial\theta_{i_{n+1}}}
\end{align*}
\end{lem}

\begin{prop}\label{prop1}\leavevmode\par
\begin{itemize}
\item[(1)] 
\[
\frac{\partial^n  f(x|\theta)}{\partial \theta_{i_1}\cdots \partial \theta_{i_n}}=-\frac{\frac{\partial^n p(x|\theta)}{\partial \theta_{i_1}\cdots \partial \theta_{i_n}}}{p(x|\theta)} + G_{\theta_{i_1}\cdots\theta_{i_n}}(x,\theta)
\]

\item[(2)] When the order of partial differentiation and integration can be exchanged,
\begin{align*}
\left.\frac{\partial^n K}{\partial \theta_{i_1}\cdots \partial \theta_{i_n}}\right|_{\theta=\theta_*}
=\mathbb{E}_X\left[ G_{\theta_{i_1}\cdots\theta_{i_n}}(X,\theta_*)\right]
\end{align*}
which means that the first term on the right side of (1) becomes a random variable with expected value zero at $\theta=\theta_*$.
\end{itemize}
\end{prop}

\begin{rem}\label{remIJ}
Using Proposition~\ref{prop1}, if we express the derivatives of the log-likelihood ratio function $f(x|\theta)$ up to the fourth order, defining $Y_{i_1\cdots i_n}:=\frac{\partial^n \log{p(x|\theta)}}{\partial\theta_{i_1}\cdots\partial\theta_{i_n}}$, we have:
\begin{align}
  \frac{\partial f(x|\theta)}{\partial \theta_{i_1}}
  =&-\frac{\frac{\partial p(x|\theta)}{\partial \theta_{i_1}}}{p(x|\theta)}\notag\\
  \frac{\partial^2  f(x|\theta)}{\partial \theta_{i_1}\partial \theta_{i_2}}
  =&-\frac{\frac{\partial^2 p(x|\theta)}{\partial \theta_{i_1}\partial \theta_{i_2}}}{p(x|\theta)}+Y_{i_1} Y_{i_2}\label{remIJ_eq02}\\
  \frac{\partial^3  f(x|\theta)}{\partial \theta_{i_1}\partial \theta_{i_2}\partial \theta_{i_3}}
    =&-\frac{\frac{\partial^3 p(x|\theta)}{\partial \theta_{i_1}\partial \theta_{i_2}\partial \theta_{i_3}}}{p(x|\theta)}
    + Y_{i_1}Y_{i_2}Y_{i_3}+ Y_{i_1,i_2}Y_{i_3}+ Y_{i_2,i_3}Y_{i_1}+ Y_{i_3,i_1}Y_{i_2}\label{remIJ_eq03}\\
  \frac{\partial^4  f(x|\theta)}{\partial \theta_{i_1}\partial \theta_{i_2}\partial \theta_{i_3}\partial \theta_{i_4}}
    =&-\frac{\frac{\partial^4 p(x|\theta)}{\partial \theta_{i_1}\partial \theta_{i_2}\partial \theta_{i_3}\partial \theta_{i_4}}}{p(x|\theta)} 
      + Y_{i_1}Y_{i_2}Y_{i_3}Y_{i_4}\notag\\ 
    &+ Y_{i_1,i_2}Y_{i_3}Y_{i_4}+Y_{i_1,i_3}Y_{i_2}Y_{i_4}+Y_{i_1,i_4}Y_{i_2}Y_{i_3}\notag\\
    &+ Y_{i_2,i_3}Y_{i_1}Y_{i_4}+Y_{i_2,i_4}Y_{i_1}Y_{i_3}+Y_{i_3,i_4}Y_{i_1}Y_{i_2}\label{remIJ_eq04}\\
    &+ Y_{i_1,i_2}Y_{i_3,i_4}+Y_{i_1,i_3}Y_{i_2,i_4}+Y_{i_1,i_4}Y_{i_2,i_3}\notag\\
    &+ Y_{i_1,i_2,i_3}Y_{i_4}+Y_{i_1,i_2,i_4}Y_{i_3}+Y_{i_1,i_3,i_4}Y_{i_2}+Y_{i_2,i_3,i_4}Y_{i_1}\notag
\end{align}
In particular, using (\ref{remIJ_eq02}), we find:
\begin{align*}
\left.\frac{\partial^2  K(\theta)}{\partial \theta_{i_1}\partial \theta_{i_2}}\right|_{\theta=\theta_*}
=&\mathbb{E}_X\left[\left.\frac{\partial \log{p(X|\theta)}}{\partial\theta_{i_1}}\right|_{\theta=\theta_*}\cdot\left.\frac{\partial \log{p(X|\theta)}}{\partial\theta_{i_2}}\right|_{\theta=\theta_*}\right]\\
=&\mathrm{Cov}\left(\left.\frac{\partial \log{p(X|\theta)}}{\partial\theta_{i_1}}\right|_{\theta=\theta_*}\cdot\left.\frac{\partial  \log{p(X|\theta)}}{\partial\theta_{i_2}}\right|_{\theta=\theta_*}\right)
\end{align*}
Thus, under the conditions of this paper, the Fisher information matrix $I$ and $J$ coincide.
\end{rem}

Based on Proposition~\ref{prop1}(2), it is necessary to determine the expectation of the random variable $G(X,0)$ when considering the derivatives of $K(\theta)$. Below, we provide specific calculations for $G(X,0)$ under Assumption~\ref{mainass}(1)(2).

\begin{prop}\label{prop:taylor}
  Assume that Assumption~\ref{mainass}(1)(2) is satisfied. That is, for any non-negative integers $(i_{r+1},\ldots,i_{d})$ that satisfy $i_{r+1}+\cdots +i_{d}\leq m-1$, the log-likelihood ratio function $f(X|\theta)$ satisfies:
  \[
  \left.\frac{\partial^{i_{r+1}+\cdots +i_{d}}f(X|\theta)}{\partial\theta_{r+1}^{i_{r+1}}
  \cdots\partial\theta_d^{i_{d}}}\right|_{(\theta_1,\ldots,\theta_d)=0}=0 ~~\text{a.s.}
  \]
  Under this condition, the following holds:
  \begin{itemize}
    \item [(1)] For $i_{r+1}+\cdots +i_{d}\leq 2m-1$,
      \[
      G_{\theta_{r+1}^{i_{r+1}}\cdots\theta_d^{i_d}}(X,0)=0 ~~\text{a.s.}
      \]
    \item [(2)] For $1\leq j \leq r,~i_{r+1}+\cdots +i_{d}\leq m-1$,
      \[
      G_{\theta_j\theta_{r+1}^{i_{r+1}}\cdots\theta_d^{i_d}}(X,0)=0 ~~\text{a.s.}
      \]
    \item [(3)] For $1\leq j \leq r,~i_{r+1}+\cdots +i_{d}= m$,
      \[
      G_{\theta_j\theta_{r+1}^{i_{r+1}}\cdots\theta_{d}^{i_{d}}}(X,0)
      =\left.\frac{\partial f(X|\theta)}{\partial\theta_j}\right|_{(\theta_1,\ldots,\theta_d)=0}
      \cdot
      \left.\frac{\partial^{m} f(X|\theta)}{\partial\theta^{i_{r+1}}_{r+1}\cdots\partial\theta^{i_{d}}_d}\right|_{(\theta_1,\ldots,\theta_d)=0} ~~\text{a.s.}
      \]
    \item [(4)] For $i_{r+1}+\cdots +i_{d}= 2m$,
      \begin{align*}
        &G_{\theta_{r+1}^{i_{r+1}}\cdots\theta_d^{i_d}}(X,0)\\
        =&\frac{1}{2}\sum_{\substack{j_{r+1}+\cdots+j_d=m\\ i_h\ge j_h\ge0}}
          \binom{i_{r+1}}{j_{r+1}}\cdots\binom{i_d}{j_d}
          \left.\frac{\partial^m f(X|\theta)}{\partial\theta_{r+1}^{j_{r+1}}\cdots\partial\theta_d^{j_d}}\right|_{\theta=0}
          \cdot
          \left.\frac{\partial^m f(X|\theta)}{\partial\theta_{r+1}^{i_{r+1}-j_{r+1}}\cdots\partial\theta_d^{i_d-j_d}}\right|_{\theta=0}  ~~\text{a.s.}
      \end{align*}    
  \end{itemize}
\end{prop}

\begin{ex}
  Consider the case of $d=2, m=2$ in Proposition~\ref{prop:taylor}(1)(4). 
  Suppose for any non-negative integer pair $(i_1, i_2)$ satisfying $i_1 + i_2 \leq 1$, the log-likelihood function $f(X|\theta)$ satisfies
  \[
  \left.\frac{\partial^{i_1 + i_2} f(X|\theta)}{\partial\theta_1^{i_1} \partial\theta_2^{i_2}}\right|_{(\theta_1, \theta_2)=0} = 0 \quad \text{a.s.}
  \]
  Specifically, assume
  \begin{gather}\label{ex_eq01}
  \left.\frac{\partial f(X|\theta)}{\partial\theta_1}\right|_{(\theta_1, \theta_2)=0} =
  \left.\frac{\partial f(X|\theta)}{\partial\theta_2}\right|_{(\theta_1, \theta_2)=0} = 0 \quad \text{a.s.}
  \end{gather}

  For $(i_1, i_2) = (1, 2)$ in Proposition~\ref{prop:taylor}(1), we expect
  \begin{align*}
    G_{\theta_1\theta_2^{2}}(X,0) = 0 \quad \text{a.s.}
  \end{align*}
  Indeed, from Remark~\ref{remIJ}(\ref{remIJ_eq03}),
  \begin{align*}
    &G_{\theta_1\theta_2^{2}}(X,\theta)\\
    =& \frac{\partial \log p(X|\theta)}{\partial\theta_1}\left(\frac{\partial \log p(X|\theta)}{\partial\theta_2}\right)^2
    + 2\frac{\partial^2 \log p(X|\theta)}{\partial\theta_1\partial\theta_2}\cdot\frac{\partial \log p(X|\theta)}{\partial\theta_2}\\
    &+ \frac{\partial^2 \log p(X|\theta)}{\partial\theta_2^2}\cdot\frac{\partial \log p(X|\theta)}{\partial\theta_1}
  \end{align*}
  where each term includes a first derivative of $f$, which are all zero at $\theta=0$ by (\ref{ex_eq01}), confirming that indeed $G_{\theta_1\theta_2^{2}}(X,0) = 0$ a.s.

  Next, for $(i_1, i_2) = (1, 3)$ in Proposition~\ref{prop:taylor}(4),
  \begin{align*}
    &G_{\theta_1\theta_2^{3}}(X,0)\\
    =& \frac{1}{2}\sum_{\substack{j_1+j_2=2\\ i_h\ge j_h\ge0}}
      \binom{1}{j_1}\binom{3}{j_2}
      \left.\frac{\partial^2 f(X|\theta)}{\partial\theta_1^{j_1}\partial\theta_2^{j_2}}\right|_{\theta=0}
      \cdot
      \left.\frac{\partial^2 f(X|\theta)}{\partial\theta_1^{1-j_1}\partial\theta_2^{3-j_2}}\right|_{\theta=0} \text{a.s.}\\
    =& \frac{3}{2}
      \left.\frac{\partial^2 f(X|\theta)}{\partial\theta_2^2}\right|_{\theta=0}
      \cdot
      \left.\frac{\partial^2 f(X|\theta)}{\partial\theta_1\partial\theta_2}\right|_{\theta=0}
      + \frac{3}{2}
      \left.\frac{\partial^2 f(X|\theta)}{\partial\theta_1\partial\theta_2}\right|_{\theta=0}
      \cdot
      \left.\frac{\partial^2 f(X|\theta)}{\partial\theta_2^2}\right|_{\theta=0} \text{a.s.}\\
    =& 3\left.\frac{\partial^2 f(X|\theta)}{\partial\theta_2^2}\right|_{\theta=0}
      \cdot
      \left.\frac{\partial^2 f(X|\theta)}{\partial\theta_1\partial\theta_2}\right|_{\theta=0} \text{a.s.}
  \end{align*}
  which should hold true. Indeed, considering only terms involving second derivatives as per Remark~\ref{remIJ}(\ref{remIJ_eq04}),
  \begin{align*}
    G_{\theta_1\theta_2^{3}}(X,0)
    = \left.\frac{\partial^2 f(X|\theta)}{\partial\theta_2^2}\right|_{\theta=0}
      \cdot
      \left.\frac{\partial^2 f(X|\theta)}{\partial\theta_1\partial\theta_2}\right|_{\theta=0} \times 3 \quad \text{a.s.}
  \end{align*}
  confirming that indeed Proposition~\ref{prop:taylor}(4) holds.
\end{ex}

\begin{proof}[\textbf{Proof of Main Theorem~\ref{mainthm}}]\leavevmode\par
  There are five points to be proven regarding the terms appearing in the Taylor expansion of \(K(\theta)\) at \(\theta=0\):
  \begin{itemize}
    \item [(1)] The coefficients of the terms of order 0 in \((\theta_1,\ldots,\theta_r)\) and order up to \(2m-1\) in \((\theta_{r+1},\ldots,\theta_d)\) are zero.
    \item [(2)] The coefficients of the terms of order 1 in \((\theta_1,\ldots,\theta_r)\) and order up to \(m-1\) in \((\theta_{r+1},\ldots,\theta_d)\) are zero.
    \item [(3)] The terms of order 1 in \((\theta_1,\ldots,\theta_r)\) and order \(m\) in \((\theta_{r+1},\ldots,\theta_d)\) are represented by:
      \[
      \mathbb{E}_X\left[F_1(X|\theta_1,\ldots,\theta_r)F_m(X|\theta_{r+1},\ldots,\theta_d)\right]
      \]  
    \item [(4)] The terms of order 0 in \((\theta_1,\ldots,\theta_r)\) and order \(2m\) in \((\theta_{r+1},\ldots,\theta_d)\) are represented by:
      \[
      \frac{1}{2}\mathbb{E}_X\left[F^2_m(X|\theta_{r+1},\ldots,\theta_d)\right]
      \]  
    \item [(5)] The terms of order 2 in \((\theta_1,\ldots,\theta_r)\) and order 0 in \((\theta_{r+1},\ldots,\theta_d)\) are represented by:
      \[
      \frac{1}{2}\mathbb{E}_X\left[F^2_1(X|\theta_1,\ldots,\theta_r)\right]
      \]
  \end{itemize}

  Assuming the interchange of differentiation and integration, using Proposition~\ref{prop1}(2), for any tuple of non-negative integers \((i_1,\ldots,i_d)\), the coefficient of the term \(\theta_1^{i_1}\cdots\theta_d^{i_d}\) in the Taylor expansion of \(K(\theta)\) at \(\theta=0\) can be expressed as
  \begin{equation}\label{mainthm_eq01}
  \frac{1}{i_1!\cdots i_d!}\left.\frac{\partial^{i_1+\cdots+i_d}K(\theta)}{\partial\theta_1^{i_1}\cdots\partial\theta_d^{i_d}}\right|_{(\theta_1,\ldots,\theta_d)=0}
  =\frac{1}{i_1!\cdots i_d!}\mathbb{E}_X\left[
  G_{\theta_1^{i_1}\cdots\theta_d^{i_{d}}}(X,0)
  \right]
  \end{equation}

  First, concerning points (1) and (2), according to Proposition~\ref{prop:taylor}(1)(2), the right-hand side $G$ is zero as a random variable almost surely. Therefore, (\ref{mainthm_eq01}) = 0. Thus, the theorem is demonstrated.

  Next, for point (3), as per Proposition~\ref{prop:taylor}(3) for $1 \leq j \leq r$ and $i_{r+1} + \cdots + i_{d} = m$,
  \[
  G_{\theta_j\theta_{r+1}^{i_{r+1}}\cdots\theta_{d}^{i_{d}}}(X,0)
  = \left.\frac{\partial f(X|\theta)}{\partial\theta_j}\right|_{(\theta_1,\ldots,\theta_d)=0}
  \cdot
  \left.\frac{\partial^{m} f(X|\theta)}{\partial\theta^{i_{r+1}}_{r+1}\cdots\partial\theta^{i_{d}}_d}\right|_{(\theta_1,\ldots,\theta_d)=0} \quad \text{a.s.}
\]
  From (\ref{mainthm_eq01}), the term in question is given by,
  \begin{align*}
  &\sum_{j=1}^r\sum_{\substack{i_{r+1} + \cdots + i_{d} = m \\ i_h \ge 0}}\frac{1}{i_{r+1}! \cdots i_{d}!}
    \left.\frac{\partial^{m+1} K(\theta)}
    {\partial\theta_j\partial\theta_{r+1}^{i_{r+1}}\cdots\partial\theta_{d}^{i_{d}}}\right|_{(\theta_1,\ldots,\theta_d)=0}
    \times\theta_j\theta_{r+1}^{i_{r+1}}\cdots\theta_{d}^{i_{d}}\\
  &= \sum_{j=1}^r\sum_{\substack{i_{r+1} + \cdots + i_{d} = m \\ i_h \ge 0}}
    \frac{\theta_j\theta_{r+1}^{i_{r+1}}\cdots\theta_{d}^{i_{d}}}{i_{r+1}! \cdots i_{d}!}
    \mathbb{E}_X\left[
    \left.\frac{\partial f(X|\theta)}{\partial\theta_j}\right|_{\theta=0}
    \cdot\left.\frac{\partial^{m} f(X|\theta)}{\partial\theta^{i_{r+1}}_{r+1}\cdots\partial\theta^{i_{d}}_d}\right|_{\theta=0}\right]\\
  &= \mathbb{E}_X\left[F_1(X|\theta_1,\ldots,\theta_r)F_m(X|\theta_{r+1},\ldots,\theta_d)\right]
  \end{align*}
  Thus, the expression can be represented as shown.

  (5) is a special case of (4) (\(m=1\)), so the proof is completed by demonstrating (4). By Proposition~\ref{prop:taylor}(4), for \(i_{r+1} + \cdots + i_d = 2m\),
  \begin{align*}
    &G_{\theta_{r+1}^{i_{r+1}}\cdots\theta_d^{i_d}}(X,0)\\
    =&\frac{1}{2}\sum_{\substack{j_{r+1}+\cdots+j_d=m\\ i_h\ge j_h\ge0}}
      \binom{i_{r+1}}{j_{r+1}}\cdots\binom{i_d}{j_d}
      \left.\frac{\partial^m f(X|\theta)}{\partial\theta_{r+1}^{j_{r+1}}\cdots\partial\theta_d^{j_d}}\right|_{\theta=0}
      \cdot
      \left.\frac{\partial^m f(X|\theta)}{\partial\theta_{r+1}^{i_{r+1}-j_{r+1}}\cdots\partial\theta_d^{i_d-j_d}}\right|_{\theta=0}  ~~\text{a.s.}
  \end{align*}   
  Thus, from (\ref{mainthm_eq01}),
  \begin{align*}
    &\left.\frac{\partial^{2m} K(\theta)}{\partial\theta_{r+1}^{i_{r+1}}\cdots\partial\theta_d^{i_d}}\right|_{\theta=0}\\
    =&\frac{1}{2}\sum_{\substack{j_{r+1}+\cdots+j_d=m\\ i_h\ge j_h\ge0}}\binom{i_{r+1}}{j_{r+1}}\cdots\binom{i_d}{j_d}
      \mathbb{E}_X\left[\left.\frac{\partial^m f(X|\theta)}{\partial\theta_{r+1}^{j_{r+1}}\cdots\partial\theta_d^{j_d}}\right|_{\theta=0}
      \cdot\left.\frac{\partial^m f(X|\theta)}{\partial\theta_{r+1}^{i_{r+1}-j_{r+1}}\cdots\partial\theta_d^{i_d-j_d}}\right|_{\theta=0}\right]
  \end{align*}
  Using this, we derive from the target formula:
  \begin{align*}
  &\frac{1}{2}\mathbb{E}_X\left[F_m^2(X|\theta_{r+1},\ldots,\theta_d)\right]\\
  =&\frac{1}{2}\mathbb{E}_X\left[
    \left(\sum_{\substack{k_{r+1}+\cdots +k_d=m\\ k_h\ge0}}\frac{\theta_{r+1}^{k_{r+1}}\cdots\theta_d^{k_d}}{k_{r+1}!\cdots k_d!}
    \left.\frac{\partial^m f(X|\theta)}{\partial\theta_{r+1}^{k_{r+1}}\cdots\partial\theta_d^{k_d}}\right|_{\theta=0}\right)^2
    \right]\\
  =&\frac{1}{2}
    \sum_{\substack{k_{r+1}+\cdots +k_d=m\\ k_h\ge0}}
    \sum_{\substack{l_{r+1}+\cdots +l_d=m\\ l_h\ge0}}
    \frac{\theta_{r+1}^{k_{r+1}+l_{r+1}}\cdots\theta_d^{k_d+l_d}}{k_{r+1}!\cdots k_d!l_{r+1}!\cdots l_d!}\\
    &~~~\times
    \mathbb{E}_X\left[
    \left.\frac{\partial^m f(X|\theta)}{\partial\theta_{r+1}^{k_{r+1}}\cdots\partial\theta_d^{k_d}}\right|_{\theta=0}
    \left.\frac{\partial^m f(X|\theta)}{\partial\theta_{r+1}^{l_{r+1}}\cdots\partial\theta_d^{l_d}}\right|_{\theta=0}
    \right]\\
  =&\frac{1}{2}
    \sum_{\substack{i_{r+1}+\cdots+i_d=2m \\ i_h\ge 0}}
    \sum_{\substack{j_{r+1}+\cdots+j_d=m\\ i_h\ge j_h\ge0}}
    \frac{\theta_{r+1}^{i_{r+1}}\cdots\theta_d^{i_d}}{j_{r+1}!\cdots j_d!(i_{r+1}-j_{r+1})!\cdots (i_d-j_d)!}\\
    &~~~\times
    \mathbb{E}_X\left[
    \left.\frac{\partial^m f(X|\theta)}{\partial\theta_{r+1}^{j_{r+1}}\cdots\partial\theta_d^{j_d}}\right|_{\theta=0}
    \cdot
    \left.\frac{\partial^m f(X|\theta)}{\partial\theta_{r+1}^{i_{r+1}-j_{r+1}}\cdots\partial\theta_d^{i_d-j_d}}\right|_{\theta=0}
    \right]\\
  =&\sum_{\substack{i_{r+1}+\cdots+i_d=2m \\ i_h\ge 0}}
    \frac{1}{i_{r+1}!\cdots i_d!}\cdot
    \frac{1}{2}
    \sum_{\substack{j_{r+1}+\cdots+j_d=m\\ i_h\ge j_h\ge0}}
    \binom{i_{r+1}}{j_{r+1}}\cdots\binom{i_d}{j_d}\\
  &~~~
    \times
    \mathbb{E}_X\left[
    \left.\frac{\partial^m f(X|\theta)}{\partial\theta_{r+1}^{j_{r+1}}\cdots\partial\theta_d^{j_d}}\right|_{\theta=0}
    \cdot
    \left.\frac{\partial^m f(X|\theta)}{\partial\theta_{r+1}^{i_{r+1}-j_{r+1}}\cdots\partial\theta_d^{i_d-j_d}}\right|_{\theta=0}
    \right]
    \times\theta_{r+1}^{i_{r+1}}\cdots\theta_d^{i_d}\\
  =&\sum_{\substack{i_{r+1}+\cdots+i_d=2m \\ i_h\ge 0}}
    \frac{1}{i_{r+1}!\cdots i_d!}
    \left.\frac{\partial^{2m} K(\theta)}
    {\partial\theta_{r+1}^{i_{r+1}}\cdots\partial\theta_d^{i_d}}\right|_{\theta=0}\theta_{r+1}^{i_{r+1}}\cdots\theta_d^{i_d}
  \end{align*}
  In the algebraic transformation, the variable change \(k_h + l_h = i_h, \ k_h = j_h\) (for \(h = r+1, \ldots, d\)) was performed. Notice the domains before and after the variable transformation are as follows:
  \begin{align*}
  &\{k_{r+1} + \cdots + k_d = m, \ l_{r+1} + \cdots + l_d = m, \ k_h \geq 0, l_h \geq 0\}\\
  &=\{j_{r+1} + \cdots + j_d = m, \ i_{r+1} + \cdots + i_d - (j_{r+1} + \cdots + j_d) = m, \ j_h \geq 0, i_h - j_h \geq 0\}\\
  &=\{j_{r+1} + \cdots + j_d = m, \ i_{r+1} + \cdots + i_d = 2m, \ i_h \geq j_h \geq 0\}
  \end{align*}

\end{proof}

\begin{rem}
Although it is outside the scope of this paper, semi-regularity in Main Theorem~\ref{mainthm} is not essential, and it is possible to generalize to arbitrary orders \((n,m)\) instead of \((1,m)\). Furthermore, it is also possible to generalize to tuples of three or more natural numbers.
\end{rem}

\subsection{Proof of Main Theorem~\ref{mainthm1}}
First, we state a key lemma essential to the proof of Main Theorem~\ref{mainthm1}.

\begin{lem}\label{mainlem2}\leavevmode\par
  \begin{itemize}
    \item [(1)] When Assumption~\ref{mainass}(1) is satisfied, the following equivalence holds:
      \begin{align*}
        \mathbb{E}_X\left[F^2_1(X|\theta_{1},\ldots,\theta_{r})\right]=0 &\ \Leftrightarrow\ (\theta_{1},\ldots,\theta_{r})=0
      \end{align*} 
    \item [(2)] Let \(a\) be a non-zero constant, and assume that Assumption~\ref{mainass}(1)(2)(3) is satisfied.
      \begin{itemize}
        \item [(i)] 
          When \((\theta_{r+1},\ldots,\theta_d)\) satisfies Assumption~\ref{mainass}(3)(i), the following equivalence holds:
          \begin{align*}
            \begin{split}
              &\mathbb{E}_X\left[\left\{F_1(X|\theta_{1},\ldots,\theta_{r})+a\cdot F_m(X|\theta_{r+1},\ldots,\theta_{d})\right\}^2\right]=0 \\
              \Leftrightarrow\ & (\theta_{1},\ldots,\theta_{r})=0,\ F_m(X|\theta_{r+1},\ldots,\theta_{d})=0 ~~\text{a.s.}
            \end{split}
          \end{align*}     
        \item[(ii)] 
          When \((\theta_{r+1},\ldots,\theta_d)\neq 0\) satisfies Assumption~\ref{mainass}(3)(ii), the following holds:
          \[
            \mathbb{E}_X\left[\left\{F_1(X|\theta_{1},\ldots,\theta_{r})+a\cdot F_m(X|\theta_{r+1},\ldots,\theta_{d})\right\}^2\right]>0
          \]
      \end{itemize}
    In particular, in either case (i) or (ii), if \((\theta_1,\ldots,\theta_r)\neq 0\), then
    \[
    \mathbb{E}_X\left[\left\{F_1(X|\theta_{1},\ldots,\theta_{r})+a\cdot F_m(X|\theta_{r+1},\ldots,\theta_{d})\right\}^2\right]>0
    \]
    holds.
  \end{itemize}
\end{lem}

\begin{proof}[\textbf{Proof of Main Theorem~\ref{mainthm1}}]\leavevmode\par
First, we demonstrate (\ref{eq_mainthm00}), and finally, we show (\ref{eq_mainthm01}).

When Assumption~\ref{mainass}(1)(2) is satisfied, we can use Main Theorem~\ref{mainthm} to express the Taylor expansion of $K(\theta)$ at $(\theta_1, \ldots, \theta_d) = 0$ as
\[
K(\theta) = \frac{1}{2} \mathbb{E}_X \left[ \left\{ F_1(X|\theta_1, \ldots, \theta_r) + F_m(X|\theta_{r+1}, \ldots, \theta_d) \right\}^2 \right] + \text{(higher order terms)}
\]
Here, the (higher order terms) are specifically expressed as the sum of the following four terms:

\begin{itemize}
  \item $f_0(\theta_{r+1},\ldots,\theta_d)$: terms of order $2m+1$ or higher
  \item $f_1(\theta_1,\ldots,\theta_d)$: first-degree homogeneous in $\theta_1, \ldots, \theta_r$ and of order $m+1$ or higher in $\theta_{r+1}, \ldots, \theta_d$
  \item $f_2(\theta_1,\ldots,\theta_d)$: of order at least second-degree in $\theta_1, \ldots, \theta_r$ and at least first-degree in $\theta_{r+1}, \ldots, \theta_d$
  \item $f_3(\theta_1,\ldots,\theta_r)$: third-degree or higher
\end{itemize}

In the following, we demonstrate these in the case of the lowest degree terms, namely:
\begin{itemize}
  \item $f_0(\theta_{r+1},\ldots,\theta_d)$: homogeneous of degree $2m+1$
  \item $f_1(\theta_1,\ldots,\theta_d)$: first-degree homogeneous in $\theta_1, \ldots, \theta_r$ and homogeneous of degree $m+1$ in $\theta_{r+1}, \ldots, \theta_d$
  \item $f_2(\theta_1,\ldots,\theta_d)$: second-degree homogeneous in $\theta_1, \ldots, \theta_r$ and first-degree homogeneous in $\theta_{r+1}, \ldots, \theta_d$
  \item $f_3(\theta_1,\ldots,\theta_r)$: homogeneous of third-degree
\end{itemize}
The general case is proven similarly.

We will consider the real log canonical threshold by performing the following blow-ups:
\begin{itemize}
  \item [(a)] Perform a blow-up centered at the origin of $\mathbb{R}^d$ once.
  \item [(b)] If the exceptional surface in (a) is $\{\theta_i = 0\}$ (where $i = r+1, \ldots, d$), perform another blow-up centered at the subvariety
               $\{(\theta_1, \ldots, \theta_d) \mid \theta_1 = \cdots = \theta_r = \theta_i = 0\}$.
  \item [(c)] If the exceptional surface in (b) is $\{\theta_i = 0\}$, repeat (b) until the total number of blow-ups reaches $m$.
\end{itemize}

Let's first consider (a).

(a-1) Consider the case where the exceptional surface is $\{\theta_i = 0\} (i=1, \ldots, r)$. For example, in the case $i=1$, that is, when we perform a blow-up with $\{\theta_2 = \theta_1 \theta_2^{\prime}, \ldots, \theta_d = \theta_1 \theta_d^{\prime}\}$, the exceptional surface is $\{\theta_1 = 0\}$, and
\begin{align*}
K(\theta) = &\frac{1}{2} \mathbb{E}_X \left[ \left\{ F_1(X|\theta_1, \ldots, \theta_r) + F_m(X|\theta_{r+1}, \ldots, \theta_d) \right\}^2 \right] \\
&\ \ \ \ \ \ + f_0(\theta_{r+1}, \ldots, \theta_d) + f_1(\theta_1, \ldots, \theta_d) + f_2(\theta_1, \ldots, \theta_d) + f_3(\theta_1, \ldots, \theta_r) \\
= &\frac{1}{2} \theta_1^2 \Biggl\{ \mathbb{E}_X \left[ \left\{ F_1(X|1, \theta^{\prime}_2, \ldots, \theta^{\prime}_r) + \theta_1^{m-1} F_m(X|\theta^{\prime}_{r+1}, \ldots, \theta^{\prime}_{d}) \right\}^2 \right] \\
&\ \ \ \ \ \ + \theta_1^{2m-1} f_0(\theta^{\prime}_{r+1}, \ldots, \theta^{\prime}_d)
+ \theta_1^m f_1(1, \theta^{\prime}_2, \ldots, \theta^{\prime}_d) \\
&\ \ \ \ \ \ + \theta_1 f_2(1, \theta^{\prime}_2, \ldots, \theta^{\prime}_d)
+ \theta_1 f_3(1, \theta^{\prime}_2, \ldots, \theta^{\prime}_r) \Biggr\} \\
= &\frac{1}{2} \theta_1^2 a(\theta_1, \theta^{\prime}_2, \ldots, \theta^{\prime}_d)
\end{align*}
(where $a$ is an analytic function). Considering the point on $U_0 = g^{-1}(O)$ in this local coordinate system $(\theta_1, \theta^{\prime}_2, \ldots, \theta^{\prime}_d)$, which satisfies $\theta_1 = 0$,
\[
a(0, \theta^{\prime}_2, \ldots, \theta^{\prime}_d) =
\begin{cases*}
\mathbb{E}_X \left[ F_1^2(X|1, \theta^{\prime}_2, \ldots, \theta^{\prime}_r) \right] & $(m \geq 2)$ \\
\mathbb{E}_X \left[ \left\{ F_1(X|1, \theta^{\prime}_2, \ldots, \theta^{\prime}_r) + F_m(X|\theta^{\prime}_{r+1}, \ldots, \theta^{\prime}_d) \right\}^2 \right] & $(m = 1)$
\end{cases*}
\]
and by Lemma~\ref{mainlem2}(1)(2), we obtain $a(0, \theta^{\prime}_2, \ldots, \theta^{\prime}_d) > 0$. Therefore, in this local coordinate system, the normal crossing of $K(\theta)$ is obtained at any point $Q$ on $U_0$, and
  \[
  \inf_{Q \in U_0} \min_{i=1, \ldots, d} \frac{h_i^{(Q)} + 1}{k_i^{(Q)}} = \frac{d-1+1}{2} = \frac{d}{2}
  \]
(multiplicity is 1).

(a-2) Next, consider the case where the exceptional surface is $\{\theta_i = 0\} (i=r+1, \ldots, d)$. For example, in the case $i=d$, that is, when we perform a blow-up with $\{\theta_1 = \theta_d \theta_1^{\prime}, \ldots, \theta_{d-1} = \theta_d \theta_{d-1}^{\prime}\}$, the exceptional surface is $\{\theta_d = 0\}$, and
\begin{align*}
K(\theta) = &\frac{1}{2} \mathbb{E}_X \left[ \left\{ F_1(X|\theta_1, \ldots, \theta_r) + F_m(X|\theta_{r+1}, \ldots, \theta_d) \right\}^2 \right] \\
&\ \ \ \ \ \ + f_0(\theta_{r+1}, \ldots, \theta_d) + f_1(\theta_1, \ldots, \theta_d) + f_2(\theta_1, \ldots, \theta_d) + f_3(\theta_1, \ldots, \theta_r) \\
= &\frac{1}{2} \theta_d^2 \Biggl\{ \mathbb{E}_X \left[ \left\{ F_1(X|\theta^{\prime}_1, \ldots, \theta^{\prime}_r) + \theta_d^{m-1} F_m(X|\theta^{\prime}_{r+1}, \ldots, \theta^{\prime}_{d-1}, 1) \right\}^2 \right] \\
&\ \ \ \ \ \ + \theta_d^{2m-1} f_0(\theta^{\prime}_{r+1}, \ldots, \theta^{\prime}_{d-1}, 1)
+ \theta_d^m f_1(\theta^{\prime}_1, \ldots, \theta^{\prime}_{d-1}, 1) \\
&\ \ \ \ \ \ + \theta_d f_2(\theta^{\prime}_1, \ldots, \theta^{\prime}_{d-1}, 1)
+ \theta_d f_3(\theta^{\prime}_1, \ldots, \theta^{\prime}_r) \Biggr\} \\
= &\frac{1}{2} \theta_d^2 a(\theta^{\prime}_1, \ldots, \theta^{\prime}_{d-1}, \theta_d)
\end{align*}
(where $a$ is an analytic function). Considering the point on $U_0 = g^{-1}(O)$ in this local coordinate system $(\theta^{\prime}_1, \ldots, \theta^{\prime}_{d-1}, \theta_d)$, which satisfies $\theta_d = 0$,
\[
a(\theta^{\prime}_1, \ldots, \theta^{\prime}_{d-1}, 0) =
\begin{cases*}
\mathbb{E}_X \left[ F_1^2(X|\theta^{\prime}_1, \ldots, \theta^{\prime}_r) \right] & $(m \geq 2)$ \\
\mathbb{E}_X \left[ \left\{ F_1(X|\theta^{\prime}_1, \ldots, \theta^{\prime}_r) + F_m(X|\theta^{\prime}_{r+1}, \ldots, \theta^{\prime}_{d-1}, 1) \right\}^2 \right] & $(m = 1)$
\end{cases*}
\]
holds. For any point $Q$ on $U_0$ that satisfies $(\theta^{\prime}_1, \ldots, \theta^{\prime}_r) \neq 0$, by Lemma~\ref{mainlem2}(1)(2), the normal crossing of $K(\theta)$ is obtained, and
  \[
\min_{i=1, \ldots, d} \frac{h_i^{(Q)} + 1}{k_i^{(Q)}} = \frac{d-1+1}{2} = \frac{d}{2}
  \]
(multiplicity is 1).

Next, consider the points on $U_0$ that satisfy $(\theta^{\prime}_1, \ldots, \theta^{\prime}_r) = 0$.
First, for the case of $m = 1$,
  \[
  S=
  \left\{
  (\theta^{\prime}_1, \ldots, \theta^{\prime}_d) ~\middle|~
  \begin{array}{l}
    (\theta^{\prime}_1, \ldots, \theta^{\prime}_{r}, \theta_d) = 0 \\
    F_m(X|\theta^{\prime}_{r+1}, \ldots, \theta^{\prime}_{d-1}, 1) = 0 \quad (\text{a.s.})
  \end{array}
  \right\} \subset U_0 
  \]  
For points $Q$ on $U_0$ that are not included in $S$, by Lemma~\ref{mainlem2}(2), $a(\theta^{\prime}_1, \ldots, \theta^{\prime}_{d-1}, 0) > 0$, so the normal crossing of $K(\theta)$ is obtained, and
  \[
  \inf_{Q \in U_0 \setminus S} \min_{i=1, \ldots, d} \frac{h_i^{(Q)} + 1}{k_i^{(Q)}} = \frac{d}{2}
  \]
(multiplicity is 1). Therefore, for $m = 1$, (\ref{eq_mainthm00}) is shown.

On the other hand, for $m \geq 2$, $a(\theta^{\prime}_1, \ldots, \theta^{\prime}_{d-1}, 0) = 0$, and the normal crossing of $K(\theta)$ is not obtained. Therefore, to obtain the normal crossing of $K(\theta)$, it is necessary to further blow-up centered at the subvariety
$\{(\theta^{\prime}_1, \ldots, \theta^{\prime}_{d-1}, \theta_d) \mid \theta^{\prime}_1 = \cdots = \theta^{\prime}_r = \theta_d = 0\}$.

Hereafter, we assume $m \geq 2$ and use the transformed coordinates without the notation ${}^{\prime}$. In this case, note that $r < d$.\\

For (b), in the local coordinates where the exceptional surface in (a) is $\{\theta_d = 0\}$, further blow-up is performed centered at the subvariety $\{(\theta_1, \ldots, \theta_d) \mid \theta_1 = \cdots = \theta_r = \theta_d = 0\}$. The same argument applies to other local coordinates.

(b-1) First, consider the case where the exceptional surface is $\{\theta_i = 0\} (i=1, \ldots, r)$. For example, in the case $i=1$, that is, when we perform a blow-up with $\{\theta_2 = \theta_1 \theta_2^{\prime}, \ldots, \theta_r = \theta_1 \theta_r^{\prime}, \theta_d = \theta_1 \theta_d^{\prime}\}$, the exceptional surface is $\{\theta_1 = 0\}$, and
\begin{align*}
K(\theta) = &\frac{1}{2} \theta_d^2
\Biggl\{ \mathbb{E}_X \left[ \left\{ F_1(X|\theta_1, \ldots, \theta_r) + \theta_d^{m-1} F_m(X|\theta_{r+1}, \ldots, \theta_{d-1}, 1) \right\}^2 \right] \\
&\ \ \ \ \ \ + \theta_d^{2m-1} f_0(\theta_{r+1}, \ldots, \theta_{d-1}, 1)
+ \theta_d^m f_1(\theta_1, \ldots, \theta_{d-1}, 1) \\
&\ \ \ \ \ \ + \theta_d f_2(\theta_1, \ldots, \theta_{d-1}, 1)
+ \theta_d f_3(\theta_1, \ldots, \theta_r) \Biggr\} \\
= &\frac{1}{2} \theta_1^4 \theta_d^{\prime 2} \Biggl\{ \mathbb{E}_X \left[ \left\{ F_1(X|1, \theta^{\prime}_2, \ldots, \theta^{\prime}_r) + \theta_1^{m-2} \theta_d^{\prime m-1} F_m(X|\theta_{r+1}, \ldots, \theta_{d-1}, 1) \right\}^2 \right] \\
&\ \ \ \ \ \ \ \ \ + \theta_1^{2m-3} \theta_d^{\prime 2m-1} f_0(\theta_{r+1}, \ldots, \theta_{d-1}, 1) \\
&\ \ \ \ \ \ \ \ \ + \theta_1^{m-1} \theta_d^{\prime m} f_1(1, \theta_2^{\prime}, \ldots, \theta_r^{\prime}, \theta_{r+1}, \ldots, \theta_{d-1}, 1) \\
&\ \ \ \ \ \ \ \ \ + \theta_1 \theta^{\prime}_d f_2(1, \theta_2^{\prime}, \ldots, \theta_r^{\prime}, \theta_{r+1}, \ldots, \theta_{d-1}, 1)
+ \theta_1^2 \theta^{\prime}_d f_3(1, \theta_2^{\prime}, \ldots, \theta_r^{\prime}) \Biggr\} \\
= &\frac{1}{2} \theta_1^4 \theta_d^{\prime 2} a(\theta_1, \theta^{\prime}_2, \ldots, \theta^{\prime}_r, \theta_{r+1}, \ldots, \theta_{d-1}, \theta^{\prime}_d)
\end{align*}
(where $a$ is an analytic function). Considering the point on $U_0 = g^{-1}(O)$ in this local coordinate system $(\theta_1, \theta^{\prime}_2, \ldots, \theta^{\prime}_r, \theta_{r+1}, \ldots, \theta_{d-1}, \theta^{\prime}_d)$, which satisfies $\theta_1 = 0$,
\begin{align*}
&a(0, \theta^{\prime}_2, \ldots, \theta^{\prime}_r, \theta_{r+1}, \ldots, \theta_{d-1}, \theta^{\prime}_d) \\
=&
\begin{cases*}
\mathbb{E}_X \left[ F_1^2(X|1, \theta^{\prime}_2, \ldots, \theta^{\prime}_r) \right] & $(m \geq 3)$ \\
\mathbb{E}_X \left[ \left\{ F_1(X|1, \theta^{\prime}_2, \ldots, \theta^{\prime}_r) + \theta_d^{\prime} F_m(X|\theta_{r+1}, \ldots, \theta_{d-1}, 1) \right\}^2 \right] & $(m = 2)$
\end{cases*}
\end{align*}
and by Lemma~\ref{mainlem2}(1)(2),
\[
a(0, \theta^{\prime}_2, \ldots, \theta^{\prime}_r, \theta_{r+1}, \ldots, \theta_{d-1}, \theta^{\prime}_d) > 0
\]
is obtained. Therefore, in this local coordinate system, the normal crossing of $K(\theta)$ is obtained at any point $Q$ on $U_0$, and
  \[
  \inf_{Q \in U_0} \min_{i=1, \ldots, d} \frac{h_i^{(Q)} + 1}{k_i^{(Q)}} = \min \left( \frac{d + r - 1 + 1}{4}, \frac{d}{2} \right) = \frac{d + r}{4}
  \]
(multiplicity is 1).

(b-2) Next, consider the case where the exceptional surface is $\{\theta_d = 0\}$. That is, when we perform a blow-up with $\{\theta_1 = \theta_d \theta_1^{\prime}, \ldots, \theta_r = \theta_d \theta_r^{\prime}\}$,
\begin{align*}
K(\theta) = &\frac{1}{2} \theta_d^2
\Biggl\{ \mathbb{E}_X \left[ \left\{ F_1(X|\theta_1, \ldots, \theta_r) + \theta_d^{m-1} F_m(X|\theta_{r+1}, \ldots, \theta_{d-1}, 1) \right\}^2 \right] \\
&\ \ \ \ \ \ + \theta_d^{2m-1} f_0(\theta_{r+1}, \ldots, \theta_{d-1}, 1)
+ \theta_d^m f_1(\theta_1, \ldots, \theta_{d-1}, 1) \\
&\ \ \ \ \ \ + \theta_d f_2(\theta_1, \ldots, \theta_{d-1}, 1)
+ \theta_d f_3(\theta_1, \ldots, \theta_r) \Biggr\} \\
= &\frac{1}{2} \theta_d^4
  \Biggl\{ \mathbb{E}_X \left[ \left\{
  F_1(X|\theta^{\prime}_1, \ldots, \theta^{\prime}_r) + \theta_d^{m-2} F_m(X|\theta_{r+1}, \ldots, \theta_{d-1}, 1)
  \right\}^2 \right] \\
  &\ \ \ \ \ \ + \theta_d^{2m-3} f_0(\theta_{r+1}, \ldots, \theta_{d-1}, 1) \\
  &\ \ \ \ \ \ + \theta_d^{m-1} f_1(\theta_1^{\prime}, \ldots, \theta_r^{\prime}, \theta_{r+1}, \ldots, \theta_{d-1}, 1) \\
  &\ \ \ \ \ \  + \theta_d f_2(\theta_1^{\prime}, \ldots, \theta_r^{\prime}, \theta_{r+1}, \ldots, \theta_{d-1}, 1)
  + \theta_d^2 f_3(\theta_1^{\prime}, \ldots, \theta_r^{\prime}) \Biggr\} \\
= &\frac{1}{2} \theta_d^4 a(\theta^{\prime}_1, \ldots, \theta^{\prime}_r, \theta_{r+1}, \ldots, \theta_d)
\end{align*}
(where $a$ is an analytic function). Considering the point on $U_0 = g^{-1}(O)$ in this local coordinate system $(\theta^{\prime}_1, \ldots, \theta^{\prime}_r, \theta_{r+1}, \ldots, \theta_d)$, which satisfies $\theta_d = 0$,
\begin{align*}
  &a(\theta^{\prime}_1, \ldots, \theta^{\prime}_r, \theta_{r+1}, \ldots, \theta_{d-1}, 0)\\
  =&
  \begin{cases*}
    \mathbb{E}_X \left[ F_1^2(X|\theta^{\prime}_1, \ldots, \theta^{\prime}_r) \right] & $(m \geq 3)$ \\
    \mathbb{E}_X \left[ \left\{ F_1(X|\theta^{\prime}_1, \ldots, \theta^{\prime}_r) + F_m(X|\theta_{r+1}, \ldots, \theta_{d-1}, 1) \right\}^2 \right] & $(m = 2)$
  \end{cases*}
\end{align*}
and for any point $Q$ on $\{\theta_d = 0\}$ that satisfies $(\theta^{\prime}_1, \ldots, \theta^{\prime}_r) \neq 0$, by Lemma~\ref{mainlem2}(1)(2), the normal crossing of $K(\theta)$ is obtained, and
  \[
\min_{i=1, \ldots, d} \frac{h_i^{(Q)} + 1}{k_i^{(Q)}} = \frac{d + r}{4}
  \]
(multiplicity is 1).

  Next, consider any point on $\{\theta_d = 0\}$ that satisfies $(\theta^{\prime}_1, \ldots, \theta^{\prime}_r) = 0$.
  First, for the case of $m = 2$,
  \[
    S=
    \left\{
    (\theta^{\prime}_1, \ldots, \theta^{\prime}_r, \theta_{r+1}, \ldots, \theta_d) ~\middle|~
    \begin{array}{l}
      (\theta^{\prime}_1, \ldots, \theta^{\prime}_r, \theta_d) = 0 \\
      F_m(X|\theta_{r+1}, \ldots, \theta_{d-1}, 1) = 0 \quad (\text{a.s.})
    \end{array}
    \right\} \subset U_0 
  \]  
  For points $Q$ on $U_0$ that are not included in $S$, 
  $a(\theta^{\prime}_1, \ldots, \theta_r^{\prime}, \theta_{r+1}, \ldots, \theta_{d-1}, 0) > 0$
  by Lemma~\ref{mainlem2}(2), so the normal crossing of $K(\theta)$ is obtained, and
  \[
  \inf_{Q \in U_0 \setminus S} \min_{i=1, \ldots, d} \frac{h_i^{(Q)} + 1}{k_i^{(Q)}} = \inf \left\{ \frac{d + r}{4}, \frac{d}{2} \right\} = \frac{d + r}{4}
  \]
(multiplicity is 1). Therefore, for $m = 2$, (\ref{eq_mainthm00}) is shown.

On the other hand, for $m \geq 3$, $a(\theta^{\prime}_1, \ldots, \theta_r^{\prime}, \theta_{r+1}, \ldots, \theta_{d-1}, 0) = 0$, and the normal crossing is not obtained.
Therefore, to obtain the normal crossing of $K(\theta)$, it is necessary to further blow-up centered at the subvariety
$\{(\theta^{\prime}_1, \ldots, \theta^{\prime}_r, $ $\theta_{r+1}, \ldots, \theta_d) \mid \theta^{\prime}_1 = \cdots = \theta^{\prime}_r = \theta_d = 0\}$.\\

Assuming $m \geq 3$ and using the transformed coordinates without the notation ${}^{\prime}$, we have:
\begin{align*}
K(\theta) = &\frac{1}{2}\theta_d^{4}\Biggl\{\mathbb{E}_X\left[\left\{F_1(X|\theta_{1}, \ldots, \theta_{r}) + \theta_d^{m-2} F_m(X|\theta_{r+1}, \ldots, \theta_{d-1}, 1)\right\}^2\right]\\
&\ \ \ \ \ \ + \theta_d^{2m-3} f_0(\theta_{r+1}, \ldots, \theta_{d-1}, 1)
+ \theta_d^{m-1} f_1(\theta_1, \ldots, \theta_{d-1}, 1)\\
&\ \ \ \ \ \ + \theta_d f_2(\theta_1, \ldots, \theta_{d-1}, 1)
+ \theta_d^2 f_3(\theta_1, \ldots, \theta_r)\Biggr\}
\end{align*}
In this coordinate system, we need to find the normal crossing of $K(\theta)$ at any point on the subvariety
$\{(\theta_1, \ldots, \theta_d) \mid \theta_1 = \cdots = \theta_r = \theta_d = 0\}$.\\

(c) Repeating the above discussion, after performing $m-1$ blow-ups, that is, for the initial parameters $(\theta_1, \ldots, \theta_d)$, consider the transformation
\[
\theta_1 = \theta_d^{m-1} \theta_1^{\prime}, \ldots, \theta_r = \theta_d^{m-1} \theta_r^{\prime}, \theta_{r+1} = \theta_d \theta_{r+1}^{\prime}, \ldots, \theta_{d-1} = \theta_d \theta_{d-1}^{\prime}
\]
The Jacobian of this transformation is $\theta_d^{d-r-1} \times \theta_d^{r(m-1)} = \theta_d^{d+r(m-2)-1}$, and $K(\theta)$ can be expressed as:
\begin{align*}
K(\theta) = &\frac{1}{2}\mathbb{E}_X\left[\left\{F_1(X|\theta_{1}, \ldots, \theta_{r}) + F_m(X|\theta_{r+1}, \ldots, \theta_{d})\right\}^2\right]\\
&\ \ \ \ \ \ + f_0(\theta_{r+1}, \ldots, \theta_d) + f_1(\theta_1, \ldots, \theta_d) + f_2(\theta_1, \ldots, \theta_d) + f_3(\theta_1, \ldots, \theta_r)\\
= &\frac{1}{2} \theta_d^{2(m-1)} \Biggl\{\mathbb{E}_X\left[\left\{F_1(X|\theta^{\prime}_{1}, \ldots, \theta^{\prime}_{r}) + \theta_d F_m(X|\theta^{\prime}_{r+1}, \ldots, \theta^{\prime}_{d-1}, 1)\right\}^2\right]\\
&\ \ \ \ \ \ + \theta_d^{3} f_0(\theta^{\prime}_{r+1}, \ldots, \theta^{\prime}_{d-1}, 1)
+ \theta_d^{2} f_1(\theta^{\prime}_1, \ldots, \theta^{\prime}_{d-1}, 1)\\
&\ \ \ \ \ \ + \theta_d f_2(\theta^{\prime}_1, \ldots, \theta^{\prime}_{d-1}, 1)
+ \theta_d^{m-1} f_3(\theta^{\prime}_1, \ldots, \theta^{\prime}_r)\Biggr\}
\end{align*}
In this coordinate system, perform one blow-up centered at $\{(\theta^{\prime}_1, \ldots, \theta^{\prime}_{d-1}, \theta_d) \mid \theta^{\prime}_1 = \cdots = \theta^{\prime}_r = \theta_d = 0\}$.

(c-1) First, consider the case where the exceptional surface is $\{\theta^{\prime}_i = 0\} (i=1, \ldots, r)$. For example, in the case $i=1$, that is, when
\[
\{\theta^{\prime}_2 = \theta^{\prime}_1 \theta_2^{\prime\prime}, \ldots, \theta^{\prime}_r = \theta^{\prime}_1 \theta_r^{\prime\prime}, \theta_d = \theta^{\prime}_1 \theta_d^{\prime}\}
\]
we perform a blow-up, the exceptional surface is $\{\theta^{\prime}_1 = 0\}$, and the Jacobian is
\[
\theta^{\prime r}_1 \cdot \theta_d^{d+r(m-2)-1} = \theta^{\prime d+r(m-1)-1}_1 \cdot \theta_d^{\prime d+r(m-2)-1}
\]
and
\begin{align*}
K(\theta) = &\frac{1}{2} \theta_d^{2(m-1)}
\Biggl\{ \mathbb{E}_X
\left[\left\{F_1(X|\theta^{\prime}_{1}, \ldots, \theta^{\prime}_{r}) + \theta_d F_m(X|\theta^{\prime}_{r+1}, \ldots, \theta^{\prime}_{d-1}, 1)\right\}^2\right]\\
&\ \ \ \ \ \ \ \ \ + \theta_d^{3} f_0(\theta^{\prime}_{r+1}, \ldots, \theta^{\prime}_{d-1}, 1)
+ \theta_d^{2} f_1(\theta^{\prime}_1, \ldots, \theta^{\prime}_{d-1}, 1)\\
&\ \ \ \ \ \ \ \ \ + \theta_d f_2(\theta^{\prime}_1, \ldots, \theta^{\prime}_{d-1}, 1)
+ \theta_d^{m-1} f_3(\theta^{\prime}_1, \ldots, \theta^{\prime}_r)
\Biggr\}\\
= &\frac{1}{2} \theta_1^{\prime 2m} \theta_d^{\prime 2(m-1)} \Biggl\{ \mathbb{E}_X \left[ \left\{ F_1(X|1, \theta^{\prime\prime}_{2}, \ldots, \theta^{\prime\prime}_{r}) + \theta_d^{\prime} F_m(X|\theta^{\prime}_{r+1}, \ldots, \theta^{\prime}_{d-1}, 1) \right\}^2 \right] \\
&\ \ \ \ \ \ \ \ \ 
+ \theta^{\prime}_1 \theta_d^{\prime 3} f_0(\theta^{\prime}_{r+1}, \ldots, \theta^{\prime}_{d-1}, 1) + \theta^{\prime}_1 \theta_d^{\prime 2} f_1(1, \theta_2^{\prime\prime}, \ldots, \theta_r^{\prime\prime}, \theta^{\prime}_{r+1}, \ldots, \theta^{\prime}_{d-1}, 1)\\
&\ \ \ \ \ \ \ \ \ 
+ \theta^{\prime}_1 \theta^{\prime}_d f_2(1, \theta_2^{\prime\prime}, \ldots, \theta_r^{\prime\prime}, \theta^{\prime}_{r+1}, \ldots, \theta^{\prime}_{d-1}, 1)
+ \theta_1^{\prime m} \theta^{\prime m-1}_d f_3(1, \theta_2^{\prime\prime}, \ldots, \theta_r^{\prime\prime}) \Biggr\}\\
= &\frac{1}{2} \theta_1^{\prime 2m} \theta_d^{\prime 2(m-1)} a(\theta^{\prime}_1, \theta^{\prime\prime}_2, \ldots, \theta^{\prime\prime}_r, \theta^{\prime}_{r+1}, \ldots, \theta^{\prime}_d)
\end{align*}
(where $a$ is an analytic function). Considering the point on $U_0 = g^{-1}(O)$ in this local coordinate system $(\theta^{\prime}_1, \theta^{\prime\prime}_2, \ldots, \theta^{\prime\prime}_r, \theta^{\prime}_{r+1}, \ldots, \theta^{\prime}_d)$, which satisfies $\theta^{\prime}_1 = 0$,
\begin{align*}
&a(0, \theta^{\prime\prime}_2, \ldots, \theta^{\prime\prime}_r, \theta^{\prime}_{r+1}, \ldots, \theta^{\prime}_d)\\
=& \mathbb{E}_X \left[ \left\{ F_1(X|1, \theta^{\prime\prime}_{2}, \ldots, \theta^{\prime\prime}_{r}) + \theta_d^{\prime} F_m(X|\theta^{\prime}_{r+1}, \ldots, \theta^{\prime}_{d-1}, 1) \right\}^2 \right]
\end{align*}
and by Lemma~\ref{mainlem2}(1)(2),
\[
a(0, \theta^{\prime\prime}_2, \ldots, \theta^{\prime\prime}_r, \theta^{\prime}_{r+1}, \ldots, \theta^{\prime}_d) > 0
\]
Thus, in this local coordinate system, the normal crossing of $K(\theta)$ is obtained at any point $Q$ on $U_0$, and
  \begin{align*}
  \min_{i=1, \ldots, d} \frac{h_i^{(Q)}+1}{k_i^{(Q)}} =& \min \left\{ \frac{d+r(m-1)-1+1}{2m}, \frac{d+r(m-2)-1+1}{2(m-1)} \right\}\\
  =& \frac{d-r+rm}{2m}
  \end{align*}
Therefore, considering the local coordinates obtained from the first $m-1$ blow-ups,
\[
\inf_{Q \in U_0} \min_{i=1, \ldots, d} \frac{h_i^{(Q)}+1}{k_i^{(Q)}} = \inf_{1 \leq m^{\prime} \leq m} \left\{ \frac{d-r+rm^{\prime}}{2m^{\prime}} \right\} = \frac{d-r+rm}{2m}
\]
(multiplicity is 1).

(c-2) Next, consider the case where the exceptional surface is $\{\theta_d = 0\}$. That is, when we perform a blow-up with $\{\theta^{\prime}_1 = \theta_d \theta_1^{\prime\prime}, \ldots, \theta^{\prime}_{r} = \theta_d \theta_{r}^{\prime\prime}\}$, the Jacobian is $\theta^{r}_d \cdot \theta_d^{d+r(m-2)-1} = \theta_d^{d+r(m-1)-1}$, and
\begin{align*}
K(\theta) = &\frac{1}{2} \theta_d^{2(m-1)}
\Biggl\{ \mathbb{E}_X
\left[\left\{F_1(X|\theta^{\prime}_{1}, \ldots, \theta^{\prime}_{r}) + \theta_d F_m(X|\theta^{\prime}_{r+1}, \ldots, \theta^{\prime}_{d-1}, 1)\right\}^2\right]\\
&\ \ \ \ \ \ + \theta_d^{3} f_0(\theta^{\prime}_{r+1}, \ldots, \theta^{\prime}_{d-1}, 1)
+ \theta_d^{2} f_1(\theta^{\prime}_1, \ldots, \theta^{\prime}_{d-1}, 1)\\
&\ \ \ \ \ \ + \theta_d f_2(\theta^{\prime}_1, \ldots, \theta^{\prime}_{d-1}, 1)
+ \theta_d^{m-1} f_3(\theta^{\prime}_1, \ldots, \theta^{\prime}_r)
\Biggr\}\\
= &\frac{1}{2} \theta_d^{2m}
\Biggl\{ \mathbb{E}_X
\left[\left\{F_1(X|\theta^{\prime\prime}_{1}, \ldots, \theta^{\prime\prime}_{r}) + F_m(X|\theta^{\prime}_{r+1}, \ldots, \theta^{\prime}_{d-1}, 1)\right\}^2\right]\\
&\ \ \ \ \ \ + \theta_d f_0(\theta^{\prime}_{r+1}, \ldots, \theta^{\prime}_{d-1}, 1)
+ \theta_d f_1(\theta_1^{\prime\prime}, \ldots, \theta_r^{\prime\prime}, \theta^{\prime}_{r+1}, \ldots, \theta^{\prime}_{d-1}, 1)\\
&\ \ \ \ \ \ + \theta_d f_2(\theta_1^{\prime\prime}, \ldots, \theta_r^{\prime\prime}, \theta^{\prime}_{r+1}, \ldots, \theta^{\prime}_{d-1}, 1)
+ \theta_d^m f_3(\theta_1^{\prime\prime}, \ldots, \theta_r^{\prime\prime}) \Biggr\}\\
= &\frac{1}{2} \theta_d^{2m} a(\theta^{\prime\prime}_1, \ldots, \theta^{\prime\prime}_r, \theta^{\prime}_{r+1}, \ldots, \theta^{\prime}_{d-1}, \theta_d)
\end{align*}
(where $a$ is an analytic function). Considering the point on $U_0 = g^{-1}(O)$ in this local coordinate system $(\theta^{\prime\prime}_1, \ldots, \theta^{\prime\prime}_r, \theta^{\prime}_{r+1}, \ldots, \theta^{\prime}_{d-1}, \theta_d)$, which satisfies $\theta_d = 0$,
\begin{align*}
&a(\theta^{\prime\prime}_1, \ldots, \theta^{\prime\prime}_r, \theta^{\prime}_{r+1}, \ldots, \theta^{\prime}_{d-1}, 0)\\
=&
\mathbb{E}_X
\left[
\left\{F_1(X|\theta^{\prime\prime}_1, \ldots, \theta^{\prime\prime}_{r}) + F_m(X|\theta^{\prime}_{r+1}, \ldots, \theta^{\prime}_{d-1}, 1)\right\}^2
\right]
\end{align*}
and
    \[
  S=
  \left\{
  (\theta^{\prime\prime}_1, \ldots, \theta^{\prime\prime}_r, \theta^{\prime}_{r+1}, \ldots, \theta^{\prime}_{d-1}, \theta_d)~\middle|~
  \begin{array}{l}
    (\theta^{\prime\prime}_1, \ldots, \theta^{\prime\prime}_{r}, \theta_d)=0 \\
    F_m(X|\theta^{\prime}_{r+1}, \ldots, \theta^{\prime}_{d-1}, 1)= 0 \text{ (a.s.)}
  \end{array}
  \right\} \subset U_0 \]  
At points $Q$ on $U_0$ not included in $S$, by Lemma~\ref{mainlem2}(2),
$a(\theta^{\prime\prime}_1, \ldots, \theta^{\prime\prime}_{r}, \theta^{\prime}_{r+1}, \ldots, \theta^{\prime}_{d-1}, 0)>0$, so the normal crossing of $K(\theta)$ is obtained, and considering the local coordinates obtained from the first $m-1$ blow-ups,
\[
\inf_{Q \in U_0 \setminus S} \min_{i=1, \ldots, d} \frac{h_i^{(Q)}+1}{k_i^{(Q)}}
=\inf_{1 \leq m^{\prime} \leq m} \left\{\frac{d-r+rm^{\prime}}{2m^{\prime}}\right\} = \frac{d-r+rm}{2m}
\]
(multiplicity is 1). Note that at points included in $S$, $a(\theta^{\prime\prime}_1, \ldots, \theta^{\prime\prime}_{r}, \theta^{\prime}_{r+1}, \ldots, \theta^{\prime}_{d-1}, 0) = 0$, and the normal crossing of $K(\theta)$ is not obtained.\\

Finally, we verify (\ref{eq_mainthm01}). Consider the case where all parameters $(\theta_{r+1},\ldots,\theta_d) \neq 0$ satisfy Assumption~\ref{mainass}(3)(ii), that is, $F_m(X|\theta_{r+1},\ldots,\theta_d)$ is linearly independent with the $r$ random variables
\[
\left.\frac{\partial f(X|\theta)}{\partial\theta_{1}}\right|_{(\theta_1,\ldots,\theta_d)=0},
\ldots,
\left.\frac{\partial f(X|\theta)}{\partial\theta_{r}}\right|_{(\theta_1,\ldots,\theta_d)=0}.
\]

In this scenario, for all parameters $(\theta_{r+1}, \ldots, \theta_d) \neq 0$,
\[
F_m(X|\theta_{r+1}, \ldots, \theta_d) \neq 0~~(a.s.)
\]
therefore, $S=\{\}$, and the blow-up $g$ provides normal crossings of $K(\theta)$ at all points $Q$ on $U_0 := g^{-1}(0)$, and the real log canonical threshold $\lambda_O$ at $\theta=0$ is given as follows (multiplicity is 1):
\begin{align*}
\lambda_O=\inf_{Q\in U_0}\left\{\min_{i=1,\ldots,d}{\frac{h_i^{(Q)}+1}{k_i^{(Q)}}}\right\}
= \inf_{Q\in U_0\setminus S}\left\{\min_{i=1,\ldots,d}{\frac{h_i^{(Q)}+1}{k_i^{(Q)}}}\right\}
=\frac{d-r+rm}{2m}
\end{align*}

Moreover, since the blow-up centered at origin $O$ provided normal crossings of $K(\theta)$, in the neighborhood of the origin $O$, using the notation from Theorem~\ref{resolution thm},
\begin{align*}
K^{-1}(0)&=\left\{\theta\in\Theta~|~K(\theta)=0\right\}
=\left\{g(u)~\middle|~a(u)u_1^{k_1}\cdots u_d^{k_d}=0\right\}\\
=&\left\{g(u)~|~u_1\cdots u_d=0\right\}
=\{O\}
\end{align*}
  Particularly, the blow-up $g$ satisfies the conditions of the resolution theorem (Theorem~\ref{resolution thm}).
  Thus, the theorem is proved.

\end{proof}

\begin{rem}\leavevmode\par
From the proof of Main Theorem~\ref{mainthm1}, the following can be understood.
Considering the change in the log-likelihood ratio function $f$ before and after the blow-up,
  \begin{align*}
f(X|\theta) &= F_1(X|\theta_1, \ldots, \theta_r) + F_m(X|\theta_{r+1}, \ldots, \theta_{d}) + (\text{higher order terms})\\
&= \theta_i^m \left\{ F_1(X|\theta^{\prime}_1, \ldots, \theta^{\prime}_r) + F_m(X|\theta^{\prime}_{r+1}, \ldots, \theta^{\prime}_{i-1}, 1, \theta^{\prime}_{i+1}, \ldots, \theta^{\prime}_{d}) + g(\theta_i) \right\}
  \end{align*}
  can be expressed (where $g(0) = 0$). On the other hand, $K(\theta)$ can be expressed as
  \begin{align*}
  &K(\theta)  \\
  =&\frac{1}{2} \theta_i^{2m} \left\{ \mathbb{E}_X
  \left[\left\{ F_1(X|\theta^{\prime}_1, \ldots, \theta^{\prime}_r)
    + F_m(X|\theta^{\prime}_{r+1}, \ldots, \theta^{\prime}_{i-1}, 1, \theta^{\prime}_{i+1}, \ldots, \theta^{\prime}_{d})
    \right\}^2 \right]
  + h(\theta_i) \right\}
  \end{align*}
  (where $h(0) = 0$). From this, it follows that
  \begin{align*}
  &\lim_{\theta\rightarrow 0} \frac{\mathbb{E}_X \left[ f(X|\theta)^2 \right]}{K(\theta)}\\
  = &2 \lim_{\theta_i \rightarrow 0}
    \frac{\mathbb{E}_X
    \left[\left\{ F_1(X|\theta^{\prime}_1, \ldots, \theta^{\prime}_r)
      + F_m(X|\theta^{\prime}_{r+1}, \ldots, \theta^{\prime}_{i-1}, 1, \theta^{\prime}_{i+1}, \ldots, \theta^{\prime}_{d}) + g(\theta_i) \right\}^2 \right]}
    {\mathbb{E}_X \left[\left\{ F_1(X|\theta^{\prime}_1, \ldots, \theta^{\prime}_r)
      + F_m(X|\theta^{\prime}_{r+1}, \ldots, \theta^{\prime}_{i-1}, 1, \theta^{\prime}_{i+1}, \ldots, \theta^{\prime}_{d})
      \right\}^2 \right] + h(\theta_i)}\\
  = &2 \lim_{\theta_i \rightarrow 0}
    \frac{\mathbb{E}_X
    \left[\left\{ F_1(X|\theta^{\prime}_1, \ldots, \theta^{\prime}_r)
      + F_m(X|\theta^{\prime}_{r+1}, \ldots, \theta^{\prime}_{i-1}, 1, \theta^{\prime}_{i+1}, \ldots, \theta^{\prime}_{d}) \right\}^2 \right]}
    {\mathbb{E}_X \left[\left\{ F_1(X|\theta^{\prime}_1, \ldots, \theta^{\prime}_r)
      + F_m(X|\theta^{\prime}_{r+1}, \ldots, \theta^{\prime}_{i-1}, 1, \theta^{\prime}_{i+1}, \ldots, \theta^{\prime}_{d})
      \right\}^2 \right]}\\
  = &2
  \end{align*}
  follows. This is consistent with \cite[Theorem 6.3]{watanabe1}.

\end{rem}

\newpage
\section{Conditions for Applying Main Theorem~\ref{mainthm}}\label{sec:chg}
To use Main Theorem~\ref{mainthm1}, it is necessary to satisfy Assumption~\ref{mainass} concerning the log-likelihood ratio function $f$.
To satisfy Assumption~\ref{mainass}, the variable transformation can be specifically constructed in the following order: (i)-(iv). Here, we first describe the concrete method of construction. In this method of construction, general theories about variable transformations such as Lemma~\ref{lem3} and Corollary~\ref{mainprop} will be used, but their proofs will be deferred. References to such propositions are made by citing them as [Corollary~\ref{mainprop}] at the relevant points in the discussion.
Unless otherwise confusing, the same notation $\theta$ will be used before and after the variable transformation.

\begin{proof}[\textbf{Construction of Variable Transformations to Satisfy Assumption~\ref{mainass}}]\leavevmode\par
\begin{itemize}
  \item [(i)] 
    Let \(V_1\) be the vector space over \(\mathbb{R}\) generated by the first derivatives of \(f\) with respect to \(\theta_1, \ldots, \theta_d\). From [Lemma~\ref{lem3}], the dimension of \(V_1\) is \(r\). Therefore, by suitably permuting coordinates, we can take as a basis:
\begin{equation}\label{eq:thmchg02}
\left\{
\left.\frac{\partial f}{\partial\theta_1}\right|_{(\theta_1,\ldots,\theta_d)=0},
\ldots,
\left.\frac{\partial f}{\partial\theta_r}\right|_{(\theta_1,\ldots,\theta_d)=0}
\right\}
\end{equation}
This satisfies Assumption~\ref{mainass}(1).
In the following, consider the set of \(n\)-th derivatives of \(f\) with respect to variables \(\theta_{r+1},\ldots,\theta_d\) (\(n \geq 1\)):
\[
D_n:=
\left\{
\left.
\frac{\partial^n f}{\partial\theta_{r+1}^{i_{r+1}}\cdots\partial\theta_d^{i_d}}
\right|_{(\theta_1,\ldots,\theta_d)=0}
\middle|~
i_{r+1}+\cdots+i_d=n
\right\}
\]
We denote the vector space over \(\mathbb{R}\) generated by \(D_n\) as \(W_n\).

  \item [(ii)]
    For the first derivatives of \(f\) with respect to \(\theta_{r+1}, \ldots, \theta_d\) (e.g., the derivative with respect to \(\theta_{r+1}\)), consider the linear relationship
\[
\left.\frac{\partial f}{\partial\theta_{r+1}}\right|_{(\theta_1,\ldots,\theta_d)=0} =
\sum_{k=1}^r a_k \cdot
\left.\frac{\partial f}{\partial\theta_k}\right|_{(\theta_1,\ldots,\theta_d)=0}, ~~a_k \in \mathbb{R}
\]
which defines a coordinate transformation
\begin{equation}\label{eq:thmchg03}
\theta_1^\prime := \theta_1 + a_1 \cdot \theta_{r+1},
\ldots,
\theta_r^\prime := \theta_r + a_r \cdot \theta_{r+1}
\end{equation}
Performing this transformation affects no other first derivatives, and in the new coordinates,
\[
\left.\frac{\partial f}{\partial\theta_{r+1}}\right|_{(\theta_1,\ldots,\theta_d)=0} = 0 ~~\text{a.s.}
\]
as stated in [Corollary~\ref{mainprop}]. 
By repeating these transformations, in the transformed coordinates, the vector space \(W_1\) becomes \(\{0\}\).

  \item [(iii)]
    Starting with $n=2$, if $\dim(V_1+W_n)=r$, then $V_1+W_n=V_1$, so any element of $D_n$ can be represented as a linear combination from (\ref{eq:thmchg02}). Specifically,
\[
\left.\frac{\partial^{n} f(X|\theta)}
{\partial\theta_{r+1}^{i_{r+1}}\cdots\partial\theta_{d}^{i_{d}}}\right|_{(\theta_1,\ldots,\theta_d)=0}=
\sum_{k=1}^r a_k\cdot
\left.\frac{\partial f}{\partial\theta_k}\right|_{(\theta_1,\ldots,\theta_d)=0},
~~ a_k\in\mathbb{R}
\]
Given this representation, the coordinate transformation
\begin{equation}\label{eq:thmchg04}
\begin{cases}
  \theta_1^\prime:=\theta_1+\frac{a_1}{i_{r+1}!\cdots i_d!}\theta_{r+1}^{i_{r+1}}\cdots\theta_{d}^{i_{d}}\\
  \vdots\\
  \theta_r^{\prime}:=\theta_r+\frac{a_r}{i_{r+1}!\cdots i_d!}\theta_{r+1}^{i_{r+1}}\cdots\theta_{d}^{i_{d}}
\end{cases}
\end{equation}
can be performed without changing other derivatives of order $n$ or lower, ensuring that
\[
\left.\frac{\partial^{n}f(X|\theta)}
{\partial\theta_{r+1}^{i_{r+1}}\cdots\partial\theta_{d}^{i_{d}}}
\right|_{(\theta_1,\ldots,\theta_d)=0}
=0~~\text{a.s.}
\]
[Corollary~\ref{mainprop}]. This coordinate transformation can be applied to any element of $D_n$, resulting in $W_n=\{0\}$. Next, by incrementing $n$ by 1 and repeating as long as $\dim(V_1+W_n)=r$, $W_n$ can be kept at $\{0\}$ for this $n$.

  \item [(iv)]
    Consider the case where for some $n$, $\dim(V_1+W_n)>r$. 
    Let us denote this particular $n$ as $m$. 
    Since $W_1=\cdots=W_{m-1}=\{0\}$, Assumption~\ref{mainass}(2) is satisfied. 
    Therefore, we only need to perform a coordinate transformation that satisfies Assumption~\ref{mainass}(3).

Let $r+s:=~\dim(V_1+W_m)~ (s\geq 1)$ and use elements $B_1,\ldots,B_s$ of $D_m$ to form a basis of the vector space $V_1+W_m$:
\[
\left\{
\left.\frac{\partial f}{\partial\theta_1}\right|_{(\theta_1,\ldots,\theta_d)=0},
\ldots,
\left.\frac{\partial f}{\partial\theta_r}\right|_{(\theta_1,\ldots,\theta_d)=0}
,B_1,\ldots,B_s
\right\}.
\]
Then, any element $B \in D_m\setminus\{B_1,\ldots,B_s\}$, also being an element of $V_1+W_m$, can be expressed as
\begin{align*}
B:=\left.\frac{\partial^{m} f(X|\theta)}
{\partial\theta_{r+1}^{i_{r+1}}\cdots\partial\theta_{d}^{i_{d}}}\right|_{(\theta_1,\ldots,\theta_d)=0}
=\sum_{k=1}^r a_k\cdot\left.\frac{\partial f}{\partial\theta_k}\right|_{(\theta_1,\ldots,\theta_d)=0}
+ \sum_{j=1}^s b_j\cdot B_j
\end{align*}
Performing the coordinate transformation
\begin{equation}\label{eq:thmchg05}
\begin{cases}
  \theta_1^\prime:=\theta_1+\frac{a_1}{i_{r+1}!\cdots i_d!}\theta_{r+1}^{i_{r+1}}\cdots\theta_{d}^{i_{d}}\\
  \vdots\\
  \theta_r^{\prime}:=\theta_r+\frac{a_r}{i_{r+1}!\cdots i_d!}\theta_{r+1}^{i_{r+1}}\cdots\theta_{d}^{i_{d}}
\end{cases}
\end{equation}
allows for
\[
B=\sum_{j=1}^s b_j B_j
\]
without changing other derivatives up to order $m$ [Corollary~\ref{mainprop}]. 
By repeating this for all elements of $D_m\setminus\{B_1,\ldots,B_s\}$, any element of $D_m$ can be expressed as a linear combination of $\{B_1,\ldots,B_s\}$. 
Therefore $W_m$ is generated by $\{B_1,\ldots,B_s\}$ over $\mathbb{R}$. 
Since these generators are linearly independent from the generators of $V_1$ (\ref{eq:thmchg02}), it follows that $V_1\cap W_m=\{0\}$.

We verify that in these coordinates Assumption~\ref{mainass}(3) is satisfied. 
Consider the $m$-th term $F_m$ in the Taylor expansion of $f$ around $(\theta_1,\ldots,\theta_d)=0$. 
For parameters $(\theta_{r+1},\ldots,\theta_d)$ where $F_m \neq 0$ (a.s.), it suffices to show that $F_m$ is linearly independent from (\ref{eq:thmchg02}). 
This follows from $F_m \in W_m$ and $V_1 \cap W_m = \{0\}$.
\end{itemize}
  
\end{proof}

\begin{rem}
  Ultimately, to satisfy Assumption~\ref{mainass}, it is necessary to perform the variable transformations (\ref{eq:thmchg03}), (\ref{eq:thmchg04}), and (\ref{eq:thmchg05}) defined by the linear dependency of the random variables. Furthermore, performing these coordinate transformations corresponds to establishing the conditions:
\[
  V_1 = V_1 + W_1 = V_1 + W_2 = \cdots = V_1 + W_{m-1} \subsetneq V_1 + W_m,
\]
where $W_1 = \cdots = W_{m-1} = \{0\},~ V_1 \cap W_m = \{0\}$.
\end{rem}

\begin{ex}[Case of $m=2$]
  In the statistical model (\ref{eq:exkongou01}) of Example~\ref{ex:m=2}, the variable transformation 
  \begin{equation}\label{eq:exkongou02}
  \theta_1^{\prime} = \theta_1 + \theta_2    
  \end{equation}
  was performed from the beginning. This is actually the variable transformation (\ref{eq:thmchg03}).
  In fact, it is more natural to set the statistical model as
  \[
  p(X=x |\theta_1, \theta_2) :=
  \frac{1}{2} \cdot \tilde{p}\left(X=x \middle| \theta_1 + \frac{1}{2}\right)
  + \frac{1}{2} \cdot \tilde{p}\left(X=x \middle| \theta_2 + \frac{1}{2}\right)
  \quad (x=0,1,2),
  \]
  but
  \begin{align*}
  \left. \frac{\partial f}{\partial \theta_1} \right|_{(\theta_1, \theta_2) = 0}
  =
  \begin{cases}
    2, & x=0\\
    0, & x = 1\\
    -2, & x=2
  \end{cases}
  ,~~~~
  \left. \frac{\partial f}{\partial \theta_2} \right|_{(\theta_1, \theta_2) = 0}
  =
  \begin{cases}
    2, & x=0\\
    0, & x = 1\\
    -2, & x=2
  \end{cases}
  \end{align*}
  and there is a linear dependency between them:
  \[
  \left. \frac{\partial f}{\partial \theta_2} \right|_{(\theta_1, \theta_2) = 0}
  = \left. \frac{\partial f}{\partial \theta_1} \right|_{(\theta_1, \theta_2) = 0}.
  \]
  The variable transformation (\ref{eq:thmchg03}) determined from this is exactly (\ref{eq:exkongou02}). As calculated in Example~\ref{ex:m=2}, after the variable transformation,
  \[
  \left. \frac{\partial f}{\partial \theta_2} \right|_{(\theta^{\prime}_1, \theta_2) = 0} = 0.
  \]
    It has already been confirmed in Example~\ref{ex:m=2} that the statistical model after this variable transformation satisfies Assumption~\ref{mainass} in the case of $m=2$.
\end{ex}

\subsection{On Assumption~\ref{mainass}(1)}

\begin{lem}\label{mainlem}[Characterization of the Linear Independence of Random Variables]
  For \(n\) random variables \(X_1, \ldots, X_n\), let \(\Sigma := (\mathbb{E}[X_i X_j])_{1 \leq i, j \leq n} \in \mathbb{R}^{n \times n}\).
  \begin{itemize}
    \item [(1)] \(\Sigma\) is non-negative definite, and the following two conditions are equivalent:
      \begin{itemize}
        \item [(a)] \(\Sigma\) is positive definite.
        \item [(b)] The random variables \(X_1, \ldots, X_n\) are linearly independent over \(\mathbb{R}\).
      \end{itemize}
    \item [(2)] Let \(V\) be the vector space over \(\mathbb{R}\) spanned by the random variables \(X_1, \ldots, X_n\). Then, \(\text{rank}(\Sigma) = \text{dim}(V)\).
  \end{itemize}
\end{lem}

\begin{proof}
  Let \(A(X) := (X_1, \ldots, X_n)^{\top} \in \mathbb{R}^n\), then \(\Sigma = \mathbb{E}_X\left[A(X) A(X)^{\top}\right]\).
  In the following, let \(V = \left<A(X)\right>\) denote the vector space over \(\mathbb{R}\) spanned by the random variables \(X_1, \ldots, X_n\).
  \begin{itemize}
    \item [(1)] For \(\bm{u} := (u_1, \ldots, u_n)^{\top} \in \mathbb{R}^n\),
      \[
        \bm{u}^{\top} \Sigma \bm{u} = \bm{u}^{\top} \mathbb{E}_X \left[A(X) A(X)^{\top}\right] \bm{u} = \mathbb{E}_X \left[\left\|A(X)^{\top} \bm{u}\right\|^2\right] \geq 0
      \]
      \[
        \bm{u}^{\top} \Sigma \bm{u} = 0 \ \Leftrightarrow \ A(X)^{\top} \bm{u} = 0 \ \text{a.s.}
      \]
      Therefore, \(\Sigma\) is non-negative definite and two conditions (a) and (b) are equivalent.
    \item[(2)]

    Let \(r := \text{rank}(\Sigma)\). There exists a real symmetric matrix \(P \in \mathbb{R}^{n \times n}\) such that, using a regular diagonal matrix \(D \in \mathbb{R}^{r \times r}\),
\[
  P~\mathbb{E}[A(X)A(X)^{\top}]P^{\top} = 
  \left(
  \begin{array}{c|c}
    D & \boldsymbol{0}\\\hline
    \boldsymbol{0} & \boldsymbol{0}
  \end{array}
  \right)
\]
Here, if we set \(PA(X) = \left(B(X)^{\top}, C(X)^{\top}\right)^{\top}\), where \(B(X) \in \mathbb{R}^r\) and \(C(X) \in \mathbb{R}^{n-r}\), then
\begin{align*}
  \mathbb{E}\left[B(X)B(X)^{\top}\right] = D,~
  \mathbb{E}\left[C(X)C(X)^{\top}\right] = \boldsymbol{0}
\end{align*}
From (1), \(B(X)\) is linearly independent, and \(C(X) = \boldsymbol{0}\) (a.s.) follows. Consequently,
\[
\text{dim}V
= \text{dim}\left<A(X)\right> 
= \text{dim}\left<PA(X)\right>
= \text{dim}\left<B(X), C(X)\right> = r
\]
follows.
  \end{itemize}
\end{proof}

\begin{lem}\label{lem3}
  If the rank of the Fisher information matrix $I$ is $r(>0)$, then the dimension of the $\mathbb{R}$-vector space $V_1$ generated by the first derivatives of $f$ with respect to $\theta_1, \ldots, \theta_d$ is $r$.
\end{lem}
         
\begin{proof}
  Let
  \begin{align*}
    A(X):=&\left(A_1(X), \ldots, A_d(X)\right)^{\top}\in\mathbb{R}^d,\\
    A_i(X):=&\left.\frac{\partial f(X|\theta)}{\partial\theta_{i}}\right|_{\theta=\theta_*}~(i=1, \ldots, d).
  \end{align*}
  As noted in Remark~\ref{remIJ}, since $\mathbb{E}_X\left[A(X)A(X)^{\top}\right]\in\mathbb{R}^{d\times d}$ coincides with the Fisher information matrix $I$, and the rank of this matrix is $r$. Therefore, it follows from Lemma~\ref{mainlem}(2) that $\dim V_1 = r$.\\
\end{proof}

\subsection{On Assumption~\ref{mainass}(2) and (3)}

\begin{prop}\label{lem:henkan}[Properties of Coordinate Transformation]
Let $r$ be an integer greater than or equal to 1. Fix some non-negative integers $i_{r+1}, \ldots, i_{d}$ such that $i_{r+1} + \cdots + i_d \geq 1$. For $k=1, \ldots, r$, define
\[
a_k(\theta_{r+1}, \ldots, \theta_d) := \frac{c_k}{i_{r+1}! \cdots i_d!} \cdot \theta_{r+1}^{i_{r+1}} \cdots \theta_d^{i_d}, \quad c_k \in \mathbb{R},
\]
and consider the coordinate transformation $\varphi:(\theta_1, \ldots, \theta_d) \mapsto (\theta^{\prime}_1, \ldots, \theta^{\prime}_d)$ given by
\begin{align*}
\theta^{\prime}_1 &= \theta_1 + a_1(\theta_{r+1}, \ldots, \theta_d), \ldots, \theta^{\prime}_r = \theta_r + a_r(\theta_{r+1}, \ldots, \theta_d), \\
\theta^{\prime}_{r+1} &= \theta_{r+1}, \ldots, \theta^{\prime}_d = \theta_d.
\end{align*}
This transformation satisfies the following:
\begin{itemize}
  \item [(1)] $\varphi(0) = 0$ and $|\mathrm{det} \ \varphi^{\prime}(0)| = 1$, and it is bijective and analytic.
  \item [(2)] For $j = 1, \ldots, r$, $\frac{\partial}{\partial \theta^{\prime}_j} = \frac{\partial}{\partial \theta_j}$ holds.
  \item [(3)] For any $(h_{r+1}, \ldots, h_d) \in \mathbb{Z}_{\geq 0}^{d-r}$ such that $h_{r+1} + \cdots + h_d \leq i_{r+1} + \cdots + i_d$,
  \begin{align*}
  &\left. \frac{\partial^{h_{r+1} + \cdots + h_d}}{\partial \theta_{r+1}^{\prime h_{r+1}} \cdots \partial \theta_d^{\prime h_d}} \right|_{(\theta^{\prime}_1, \ldots, \theta^{\prime}_d) = 0} \\
  =& \left. \frac{\partial^{h_{r+1} + \cdots + h_d}}{\partial \theta_{r+1}^{h_{r+1}} \cdots \partial \theta_d^{h_d}} \right|_{\theta = 0} -
  \begin{cases}
    0 & \text{if } (h_{r+1}, \ldots, h_d) \neq (i_{r+1}, \ldots, i_d), \\
    \sum_{k=1}^r c_k \cdot \left. \frac{\partial}{\partial \theta_k} \right|_{\theta = 0} & \text{if } (h_{r+1}, \ldots, h_d) = (i_{r+1}, \ldots, i_d),
  \end{cases}
  \end{align*}
  holds.
\end{itemize}
\end{prop}

\begin{proof}
  (1) is evident. (2) follows from the relationship between partial derivatives before and after the variable transformation:
\begin{alignat*}{2}
\frac{\partial}{\partial\theta^{\prime}_j}
&=\sum_{k=1}^d\frac{\partial}{\partial\theta_{k}}\cdot\frac{\partial\theta_{k}}{\partial\theta^{\prime}_j}
=\sum_{k=1}^r\frac{\partial}{\partial\theta_{k}}\cdot\frac{\partial\theta_{k}}{\partial\theta^{\prime}_j}
+\sum_{k=r+1}^d\frac{\partial}{\partial\theta_k}\cdot\delta_{k,j}\\
&=\frac{\partial}{\partial\theta_j}
-\sum_{k=1}^r\frac{\partial a_{k}}{\partial\theta_j}\cdot\frac{\partial}{\partial\theta_{k}}& &(j=r+1,\ldots,d)\\
\frac{\partial}{\partial\theta^{\prime}_j}&=\frac{\partial}{\partial\theta_j}& &(j=1,\ldots,r)
\end{alignat*}

(3) can be expressed as follows: Let \( m:=i_{r+1}+\cdots+i_d(\geq 1),\ l:=h_{r+1}+\cdots+h_{d}(\leq m) \). The derivatives are expressed as follows:
\begin{align}
  \frac{\partial^l}{\partial\theta_{r+1}^{\prime h_{r+1}}\cdots\partial\theta_{d}^{\prime h_{d}}}
  &=\prod_{j=r+1}^{d}
    \left(
    \frac{\partial}{\partial\theta_j}
    -\sum_{k=1}^r\frac{\partial a_{k}}{\partial\theta_j}\cdot\frac{\partial}{\partial\theta_{k}}
  \right)^{h_j}\notag\\
  &=\prod_{j=r+1}^{d}
    \left(
    \frac{\partial^{h_j}}{\partial\theta_j^{h_j}}
    -\sum_{k=1}^r\frac{\partial^{h_j} a_{k}}{\partial\theta^{h_j}_j}\cdot\frac{\partial}{\partial\theta_{k}}
    +\cdots
    \right)\label{eq:henkan02}\\
  &=\frac{\partial^l}{\partial\theta_{r+1}^{h_{r+1}}\cdots\partial\theta_{d}^{h_{d}}}
    -\sum_{k=1}^r
    \frac{\partial^l a_{k}}{\partial\theta_{r+1}^{h_{r+1}}\cdots\partial\theta_{d}^{h_{d}}}\cdot\frac{\partial}{\partial\theta_{k}}
    +\cdots\label{eq:henkan03}
\end{align}
and the ``\(\cdots\)'' in (\ref{eq:henkan02}) and (\ref{eq:henkan03}) represent terms derived from derivatives of \( a_{k}(\theta_{r+1},\ldots,\theta_{d}) \) up to order \( l-1 \) with respect to \( \theta_{r+1},\ldots,\theta_{d} \). 
Since \( a_{k}(\theta_{r+1},\ldots,\theta_{d}) \) is a homogeneous polynomial of degree \( m \) with respect to \( \theta_{r+1},\ldots,\theta_{d} \), the ``\(\cdots\)'' in (\ref{eq:henkan02}) and (\ref{eq:henkan03}) becomes zero at \( \theta=(\theta_1,\ldots,\theta_d)=0 \), and
\[
\left.\frac{\partial^l a_{k}}{\partial\theta_{r+1}^{h_{r+1}}\cdots\partial\theta_{d}^{h_{d}}}\right|_{\theta=0}=
\begin{cases}
  0 & (h_{r+1},\ldots,h_{d})\neq (i_{r+1},\ldots,i_{d})\\
  c_k & (h_{r+1},\ldots,h_{d})= (i_{r+1},\ldots,i_{d})
\end{cases}
\]
Therefore, the following holds:
\begin{align*}
  &\left.\frac{\partial^l}{\partial\theta_{r+1}^{\prime h_{r+1}}\cdots\partial\theta_{d}^{\prime h_{d}}}\right|_{(\theta^{\prime}_1,\ldots,\theta^{\prime}_d)=0}\\
  =&\left.\frac{\partial^l }{\partial\theta_{r+1}^{h_{r+1}}\cdots\partial\theta_{d}^{h_{d}}}\right|_{\theta=0}-
  \begin{cases}
    0 & (h_{r+1},\ldots,h_{d})\neq (i_{r+1},\ldots,i_{d})\\
    \sum_{k=1}^r c_k\cdot\left.\frac{\partial}{\partial\theta_{k}}\right|_{\theta=0} & (h_{r+1},\ldots,h_{d})= (i_{r+1},\ldots,i_{d})
  \end{cases}  
\end{align*}

\end{proof}

\begin{cor}\label{mainprop}[Coordinate transformation to satisfy Assumption~\ref{mainass}(2),(3)]
  Let $r$ be an integer satisfying $1 \leq r \leq d-1$, and let $n$ be an integer greater than or equal to 1. Define
  \begin{equation}
    D_n:=
    \left\{
      \left.\frac{
        \partial^{n} f(X|\theta)}{\partial\theta_{r+1}^{i_{r+1}}\cdots\partial\theta_{d}^{i_{d}}}\right|_{(\theta_1,\ldots,\theta_d)=0}
      \middle|~ 
      i_{r+1}+\cdots +i_d=n
    \right\}    
  \end{equation}
  Let $B$ be an element of $D_n$ defined as
  \[
  B:=
  \left.\frac{
      \partial^{n} f(X|\theta)}{\partial\theta_{r+1}^{i_{r+1}}\cdots\partial\theta_{d}^{i_{d}}}\right|_{(\theta_1,\ldots,\theta_d)=0}
  \]
  Assume that $B$ can be expressed using $B_1, \ldots, B_s \in D_n$ and real numbers $c_k, b_j~ (k=1, \ldots, r,~ j=1, \ldots, s)$ as follows:
  \[
    B
    =\sum_{k=1}^r c_k\cdot \left.\frac{\partial f(X|\theta)}{\partial\theta_{k}}\right|_{(\theta_1,\ldots,\theta_d)=0}
    +\sum_{j=1}^s b_j\cdot B_j
      ~~ \text{a.s.}
  \]
  Then, applying the coordinate transformation $\varphi:(\theta_1,\dots,\theta_d) \mapsto (\theta_1^{\prime},\ldots,\theta_d^{\prime})$ given by Proposition~\ref{lem:henkan};
  \begin{align*}
    \theta_k^{\prime}:=\theta_k+\frac{c_k}{i_{r+1}!\cdots i_d!}\theta_{r+1}^{i_{r+1}}\cdots\theta_d^{i_d}~~(k=1,\ldots,r),~~c_k\in\mathbb{R}
  \end{align*}  
  without affecting the other lower-order derivatives, the transformed $B$ (denoted as $B^{\prime}$) becomes
  \[
  B^{\prime}=\sum_{j=1}^s b_jB_j~~\text{a.s.}
  \]
\end{cor}\leavevmode\par

\begin{proof}
  We need to demonstrate the following two points.
  (Below, the inverse transformation of $\varphi$ is denoted as
  $\phi:(\theta_1^{\prime},\ldots,\theta_d^{\prime})\mapsto(\theta_1,\dots,\theta_d)$.)
  \begin{itemize}
    \item [(1)] 
      For $k=1, \ldots, r$,
    \[
      \left.\frac{\partial f(X|\phi(\theta^{\prime}))}{\partial\theta^{\prime}_{k}}\right|_{(\theta^{\prime}_1, \ldots, \theta^{\prime}_d) = 0}
      = \left.\frac{\partial f(X|\theta)}{\partial\theta_{k}}\right|_{(\theta_1, \ldots, \theta_d) = 0}
    \]
    \item[(2)] 
      For any $(h_{r+1}, \ldots, h_d) \in \mathbb{Z}^{d-r}_{\geq 0}$ such that $h_{r+1} + \cdots + h_d \leq i_{r+1} + \cdots + i_d$,
      \begin{align*}
        &\left.
        \frac{\partial^{h_{r+1} + \cdots + h_d}f(X|\phi(\theta^{\prime}))}
          {\partial\theta_{r+1}^{\prime h_{r+1}} \cdots \partial\theta_d^{\prime h_d}}
        \right|_{(\theta_1^{\prime}, \ldots, \theta_d^{\prime}) = 0} \\
        =&
        \begin{cases}
          \left.
          \frac{\partial^{h_{r+1} + \cdots + h_d}f(X|\theta)}
            {\partial\theta_{r+1}^{h_{r+1}} \cdots \partial\theta_d^{h_d}}
          \right|_{(\theta_1, \ldots, \theta_d) = 0}
          &  \text{if } (h_{r+1}, \ldots, h_d) \neq (i_{r+1}, \ldots, i_d) \\
          \sum_{j=1}^s b_jB_j~ \text{a.s.} & \text{if } (h_{r+1}, \ldots, h_d) = (i_{r+1}, \ldots, i_d)
        \end{cases}        
      \end{align*}
  \end{itemize}  

  (1) can be directly applied from Proposition~\ref{lem:henkan}(2).

  For (2), when $(h_{r+1}, \ldots, h_d) = (i_{r+1}, \ldots, i_d)$, it follows from Proposition~\ref{lem:henkan}(3) that
  \begin{align*}
    B^{\prime}
    =& \left.\frac{\partial^n f(X|\phi(\theta^{\prime}))}{\partial\theta_{r+1}^{\prime i_{r+1}} \cdots \partial\theta_{d}^{\prime i_{d}}}\right|_{(\theta^{\prime}_1, \ldots, \theta^{\prime}_d) = 0} \\
    =& \left.\frac{\partial^n f(X|\theta)}{\partial\theta_{r+1}^{i_{r+1}} \cdots \partial\theta_{d}^{i_{d}}}\right|_{(\theta_1, \ldots, \theta_d) = 0}
      - \sum_{k=1}^r c_k \cdot
      \left.\frac{\partial f(X|\theta)}{\partial\theta_{k}}\right|_{(\theta_1, \ldots, \theta_d) = 0} \\
    =& \sum_{j=1}^s b_j \cdot B_j~~ \text{a.s.}
  \end{align*}
  Additionally, when $(h_{r+1}, \ldots, h_d) \neq (i_{r+1}, \ldots, i_d)$ and $h_{r+1} + \cdots + h_d \leq n$, it follows from Proposition~\ref{lem:henkan}(3) that
  \[
    \left.\frac{\partial^{h_{r+1} + \cdots + h_d} f(X|\phi(\theta^{\prime}))}{\partial\theta_{r+1}^{\prime h_{r+1}} \cdots \partial\theta_{d}^{\prime h_d}}\right|_{(\theta^{\prime}_1, \ldots, \theta^{\prime}_d) = 0}
    = \left.\frac{\partial^{h_{r+1} + \cdots + h_d} f(X|\theta)}{\partial\theta_{r+1}^{h_{r+1}} \cdots \partial\theta_{d}^{h_d}}\right|_{(\theta_1, \ldots, \theta_d) = 0}
  \]
  holds true.
  \\
\end{proof}

\begin{rem}
  Let us consider the meaning of Corollary~\ref{mainprop}.
  For simplicity, assume \( s = 0 \) and denote the random variables \( A_k(X|\theta) \) as \( A_k \).
  According to the assumption of Corollary~\ref{mainprop}, using the real numbers \( c_k\ (k=1,\ldots,r) \), we can express
  \[
    \left.\frac{\partial^{i_{r+1}+\cdots+i_d} f(X|\theta)}{\partial\theta_{r+1}^{i_{r+1}}\cdots\partial\theta_{d}^{i_{d}}}\right|_{(\theta_1,\ldots,\theta_d)=0}
    =\sum_{k=1}^r c_k A_k~~ \text{a.s.}
  \]
  Then, according to Corollary~\ref{mainprop}(2), after the variable transformation, we have
  \[
    \left.\frac{\partial^{i_{r+1}+\cdots+i_d} f(X|\phi(\theta^{\prime}))}{\partial\theta_{r+1}^{\prime i_{r+1}}\cdots\partial\theta_{d}^{\prime i_{d}}}\right|_{(\theta^{\prime}_1, \ldots, \theta^{\prime}_d) = 0}=0~~ \text{a.s.}
  \]
  Let's interpret this from a different perspective.

  When we perform a Taylor expansion of the log-likelihood ratio function \( f \) with respect to the parameter \( \theta \) (noting that the coefficients are random variables), and group the terms by linearly independent random variables, we get:
  \begin{align*}
    f &= A_1\theta_1 + \cdots + A_r\theta_r + \frac{\sum_{k=1}^r c_k A_k}{i_{r+1}!\cdots i_d!}\theta_{r+1}^{i_{r+1}}\cdots\theta_d^{i_d} + (\text{higher order terms})\ \ \text{a.s.}\\
    &= \sum_{k=1}^r A_k \left( \theta_k + \frac{c_k}{i_{r+1}!\cdots i_d!}\theta_{r+1}^{i_{r+1}}\cdots\theta_d^{i_d} \right) + (\text{higher order terms})\ \ \text{a.s.}
  \end{align*}
  Here, if we apply the variable transformation given in Corollary~\ref{mainprop}
  \begin{align*}
    \theta_k^{\prime} := \theta_k + \frac{c_k}{i_{r+1}!\cdots i_d!}\theta_{r+1}^{i_{r+1}}\cdots\theta_d^{i_d}~~ (k=1,\ldots,r)
  \end{align*}
  then we can rewrite it as
  \begin{align*}
    f = A_1\theta_1^{\prime} + \cdots + A_r\theta_r^{\prime} + (\text{higher order terms})~~ \text{a.s.}
  \end{align*}
  This shows that the coefficient (random variable) of the term involving \(\theta_{r+1}^{i_{r+1}}\cdots\theta_d^{i_d}\) can indeed be made zero (a.s.).
\end{rem}

\newpage
\section{Example on the Derivation of the Learning Coefficient}
Using the Main Theorem~\ref{mainthm1} for semi-regular models, we calculate the learning coefficient for a specific model.

\subsection{The Case of Two Parameters}
Here, we fix an arbitrary realizable parameter \(\theta_*\) at a single point and translate it to the origin \(O\), and consider the real log canonical threshold at the origin.
We assume that the statistical model \( p(x|\theta) \) at the origin satisfies the semi-regular condition (i.e., the rank \( r > 0 \) of the Fisher information matrix),
and denote the real log canonical threshold at the origin by \(\lambda_O\).
Note that the learning coefficient is the minimum value of the real log canonical thresholds calculated for each realizable parameter.

By performing variable transformations, semi-regular models with two parameters can be classified into one of the categories shown in Table \ref{fig_2parameters}.
\begin{table}[htbp]
  \centering
  \caption{Semi-Regular Models with Two Parameters.}
  \label{fig_2parameters}
  \begin{tabular}{cccc}
  Property of \( f \) at \( \theta = 0 \) & 
  \begin{tabular}{c}  
     \(\lambda_O\) \\
     (Multiplicity) 
  \end{tabular}
  & Ideal & 
   \begin{tabular}{c}  
     Geometry of \\
     \( \Theta_* \) near \( \theta = 0 \)
   \end{tabular}
   \\ \hline
 \(\frac{\partial f}{\partial\theta_1},\frac{\partial f}{\partial\theta_2}\): lin. ind.& 1 (1) & \((\theta_1, \theta_2)\) & pt \(\{(\theta_1, \theta_2) = (0, 0)\}\) \\ \hline 
   \begin{tabular}{c}
    \(\frac{\partial f}{\partial\theta_1}, \frac{\partial^m f}{\partial\theta^m_2}\): lin. ind.\\
    \(\frac{\partial f}{\partial\theta_2} = \cdots = \frac{\partial^{m-1} f}{\partial\theta^{m-1}_2} = 0\)
   \end{tabular}
  & \(\frac{m+1}{2m}\) (1) & \((\theta_1, \theta_2^m)\) & pt \(\{(\theta_1, \theta_2) = (0, 0)\}\) \\ \hline
   \begin{tabular}{c}
    \(\frac{\partial f}{\partial\theta_1}\): lin. ind.\\
    \(\forall m, \frac{\partial^m f}{\partial\theta^m_2} = 0\)
   \end{tabular}
    & \(\frac{1}{2}\) (1) & \((\theta_1)\) & line \(\{\theta_1 = 0\}\) \\
\end{tabular}
\end{table}

The above results for the real log canonical threshold \(\lambda_O\) are all consequences of Main Theorem~\ref{mainthm1} (The third case in Table \ref{fig_2parameters} corresponds to the case where \((d,r,m)=(1,1,1)\)).
From this, we find that for a statistical model with two parameters, where the Fisher information matrix at all realizable parameters has a non-zero rank, the learning coefficient is given by:
\[
\lambda = \frac{m+1}{2m}\ ,\ m = 1, 2, \ldots, \infty
\]
(multiplicity is 1). Notably, the minimum value of the learning coefficient is 1/2 and the maximum value is 1, parameterized by \( m \in \mathbb{Z}_{\geq 1} \cup \{\infty\} \). Moreover, \( \Theta_* \) does not contain singularities.

\begin{rem}
  The real log canonical threshold at the origin for the statistical model considered in Example~\ref{ex:m=2} was 3/4.
  Applying the general theory to this statistical model with \( (d, r, m) = (2, 1, 2) \), we immediately obtain:
  \[
  \lambda_O = \frac{m+1}{2m} = \frac{3}{4}
  \]
\end{rem}

The same discussion can be applied to the cases where \( r = d \) or \( r = d - 1 \). We summarize this in the following proposition.

\begin{prop}[Case of $r \geq d-1$]\label{prop:deg2}
  For a statistical model with \( d \) parameters, where the rank \( r \) of the Fisher information matrix at all realizable parameters is \( r = d-1 \) or \( r = d \),
  the learning coefficient \(\lambda\) can be expressed using a positive integer \( m \) as follows (multiplicity is 1):
  \[
  \lambda = \frac{1 + (d-1)m}{2m}\ ,\ m = 1, 2, \ldots, \infty
  \]
  Here, \( m = \infty \) represents \(\lambda = (d-1)/2\). Notably, the minimum value of the learning coefficient is \( (d-1)/2 \) and the maximum value is \( d/2 \). Furthermore, \( \Theta_* \) does not contain singularities.
\end{prop}

\begin{rem}
In the case of non-semi-regular models, there are exceptions to the above. For example, at \((\theta_1, \theta_2) = 0\), if the random variables
\[
\left.\frac{\partial f}{\partial \theta_1^2}\right|_{(\theta_1, \theta_2) = 0},\ \left.\frac{\partial f}{\partial \theta_2^3}\right|_{(\theta_1, \theta_2) = 0}
\]
are linearly independent and all derivatives of \( f \) up to those orders are zero (e.g., \( K(\theta) = \theta_1^4 + \theta_2^6 \)),
then this case applies.
(Although methods for resolving such singularities similar to Euclid's algorithm are known, they are beyond the scope of this paper.)

However, if the model is semi-regular at other realizable parameters, the above results can be applied at those points to obtain an upper bound on the learning coefficient.
\end{rem}

\subsection{Formula for the Learning Coefficient in a Mixture Distribution Model with a Constant Mixing Ratio}
Generalizing from Example~\ref{ex:m=2}, we establish the following.

\begin{ex}\label{ex_tuika}
  Let \(M (\geq 2)\) be a constant, and let \(\tilde{p}(x|\theta)\) be a binomial distribution Bin\((M, \theta)\). Thus,
\[
\tilde{p}(X=x|\theta) = \binom{M}{x} \theta^x (1-\theta)^{M-x} \quad (x=0,1,\ldots,M)
\]

  Let $H(\geq 2)$ be the number of mixture components, and $(T_1,\ldots,T_{H-1})$ be constants satisfying $0 < T_i < 1$ and $\sum_{i=1}^{H-1} T_i \neq 1$. We consider a mixture distribution model with $H$ parameters $(\theta_1,\ldots,\theta_H)$ defined as
  \[
  p(x|\theta) := T_1\tilde{p}(x|\theta_1) + \cdots + T_{H-1}\tilde{p}(x|\theta_{H-1}) + \left(1-\sum_{i=1}^{H-1}T_i\right)\tilde{p}(x|\theta_H).
  \]
  Assuming the true distribution $q(x) = \tilde{p}(x|\theta_*)$(where \(0<\theta_* <1\) is a constant), the learning coefficient is given by
  \[
  \lambda = \frac{H+1}{4}.
  \]
\end{ex}

\begin{proof}
  From Lemma~\ref{tuika_lem1}, the realizable parameter set $\Theta_*$ consists only of the point $(\theta_1,\ldots,\theta_H)=(\theta_*,\ldots,\theta_*)$.
  We redefine the statistical model by translating the origin such that
  \[
  p(x|\theta) := T_1\tilde{p}(x|\theta_1+\theta_*) + \cdots + T_{H-1}\tilde{p}(x|\theta_{H-1}+\theta_*) + \left(1-\sum_{i=1}^{H-1}T_i\right)\tilde{p}(x|\theta_H+\theta_*).
  \]
  The learning coefficient we aim to determine is the real log canonical threshold at the origin $(\theta_1,\ldots,\theta_H)=0$.
  \begin{align*}
    \left.\frac{\partial f}{\partial\theta_i}\right|_{(\theta_1,\ldots,\theta_H)=0} &= -\frac{T_i}{q(x)}\frac{\partial\tilde{p}}{\partial\theta}(x|\theta_*) \quad (i=1,\ldots,H-1),\\
    \left.\frac{\partial f}{\partial\theta_H}\right|_{(\theta_1,\ldots,\theta_H)=0} &= -\frac{1-\sum_{i=1}^{H-1} T_i}{q(x)}\frac{\partial\tilde{p}}{\partial\theta}(x|\theta_*)
  \end{align*}
  are linearly dependent, satisfying
  \[
  \left.\frac{\partial f}{\partial\theta_i}\right|_{(\theta_1,\ldots,\theta_H)=0} = \frac{T_i}{1-\sum_{i=1}^{H-1} T_i}\left.\frac{\partial f}{\partial\theta_H}\right|_{(\theta_1,\ldots,\theta_H)=0} \quad (i=1,\ldots,H-1).
  \]
  Transforming the coordinates via
  \[
  \theta_H^{\prime} := \theta_H + \frac{\sum_{i=1}^{H-1} T_i\theta_i}{1-\sum_i T_i},
  \]
  leads to
  \begin{align}
    \left.\frac{\partial f}{\partial\theta_i}\right|_{(\theta_1,\ldots,\theta_{H-1},\theta_H^{\prime})=0} &= 0 \quad (i=1,\ldots,H-1) \notag\\
    \left.\frac{\partial f}{\partial\theta_H}\right|_{(\theta_1,\ldots,\theta_{H-1},\theta_H^{\prime})=0} &= -\frac{1-\sum_{i=1}^{H-1} T_i}{q(x)}\frac{\partial\tilde{p}}{\partial\theta}(x|\theta_*). \label{tuika_eq03}
  \end{align}
  The second derivative $F_2(x|\theta_1,\ldots,\theta_{H-1})$ is represented as
  \[
  F_2(x|\theta_1,\ldots,\theta_{H-1}) = \frac{-1}{1-\sum_i T_i}\cdot\frac{1}{2q(x)}\cdot\frac{\partial^2\tilde{p}}{\partial\theta^2}(x|\theta_*)\cdot(\theta_1,\ldots,\theta_{H-1})\Sigma(\theta_1,\ldots,\theta_{H-1})^{\top}.
  \]
  The symmetric matrix $\Sigma := (\sigma_{i, j})_{1\leq i,j\leq H-1} \in \mathbb{R}^{(H-1)\times(H-1)}$ is defined as
  \[
  \sigma_{i.j} :=
  \begin{cases}
    T_i(1-\sum_{k=1}^{H-1}T_k+T_i) & \text{if } i = j\\
    T_iT_j & \text{if } i \neq j
  \end{cases}
  \]
  Clearly,
  \begin{equation}\label{tuika_eq02}
    \frac{\partial\tilde{p}}{\partial\theta}(x|\theta_*),~
    \frac{\partial^2\tilde{p}}{\partial\theta^2}(x|\theta_*): \text{are linearly independent}
  \end{equation}
  and from Lemma~\ref{tuika_lem2}, since $\Sigma$ is a positive-definite matrix,
  \[
  \forall (\theta_1,\ldots,\theta_{H-1}) \neq 0,~~(\theta_1,\ldots,\theta_{H-1})\Sigma(\theta_1,\ldots,\theta_{H-1})^{\top} \neq 0
  \]
  and $F_2(x|\theta_1,\ldots,\theta_{H-1})$ and (\ref{tuika_eq03}) are linearly independent.
  Thus, this statistical model satisfies Assumption~\ref{mainass}(3)(ii) when $(d,r,m) = (H,1,2)$, allowing the direct application of Main Theorem~\ref{mainthm1} to compute the real log canonical threshold,
  \[
  \lambda = \frac{d-r+rm}{2m} = \frac{H+1}{4}
  \]
  (multiplicity is 1).
\end{proof}

\begin{lem}\label{tuika_lem1}
  Assume the same conditions as in Example~\ref{ex_tuika}.
  In this case, the realizable parameter set $\Theta_*$ consists solely of the point $(\theta_1,\ldots,\theta_H) = (\theta_*,\ldots,\theta_*)$.
\end{lem}

\begin{proof}
  Given
  \begin{equation}\label{eq_tuika01}
  T_1\tilde{p}(x|\theta_1) + \cdots + T_{H-1}\tilde{p}(x|\theta_{H-1}) + \left(1 - \sum_{i=1}^{H-1}T_i\right)\tilde{p}(x|\theta_H) = \tilde{p}(x|\theta_*)
  \end{equation}
  Multiplying both sides of (\ref{eq_tuika01}) by $x$ and summing for $x = 0, 1, \ldots,M$ yields,
  \[
  T_1\theta_1 + \cdots + T_{H-1}\theta_{H-1} + \left(1 - \sum_{i=1}^{H-1}T_i\right)\theta_H = \theta_*
  \]
  Further, multiplying (\ref{eq_tuika01}) by $x^2$ and summing for $x = 0, 1, \ldots,M$ results in,
  \[
  T_1\theta_1^2 + \cdots + T_{H-1}\theta_{H-1}^2 + \left(1 - \sum_{i=1}^{H-1}T_i\right)\theta_H^2 = \theta_*^2
  \]
  Thus,
  \begin{align*}
  &T_1\theta_1^2 + \cdots + T_{H-1}\theta_{H-1}^2 + \left(1 - \sum_{i=1}^{H-1}T_i\right)\theta_H^2 \\
  = &\left(T_1\theta_1 + \cdots + T_{H-1}\theta_{H-1} + \left(1 - \sum_{i=1}^{H-1}T_i\right)\theta_H\right)^2
  \end{align*}
  and, since the equality condition of Jensen's inequality holds, it follows that $\theta_1 = \cdots = \theta_{H-1} = \theta_H=\theta_*$.
\end{proof}

\begin{lem}\label{tuika_lem2}
  Let $N$ be an integer greater than or equal to 1, and variables $T_1,\ldots,T_N (0 \leq T_i \leq 1, \sum_{i=1}^N T_i \leq 1)$. Define the symmetric matrix $\Sigma_N := (\sigma_{i,j})_{1\leq i,j \leq N} \in \mathbb{R}^{N\times N}$ as
  \[
  \sigma_{i.j} :=
  \begin{cases}
    T_i(1-\sum_{k=1}^{N}T_k+T_i) & i = j\\
    T_iT_j & i \neq j
  \end{cases}
  \]
  Then, $\Sigma_N$ is non-negative definite, and the following two conditions are equivalent:
  \begin{itemize}
    \item [(i)] $\forall i=1,\ldots,N, ~ 0<T_i<1,~ \sum_{i=1}^N T_i \neq 1$
    \item [(ii)] $\Sigma_N$ is positive-definite
  \end{itemize}
\end{lem}

\begin{proof}
  A simple calculation gives
  \[
  \text{det }\Sigma_N = \left(1-\sum_{i=1}^N T_i\right)\prod_{i=1}^N T_i,
  \]
  hence, if we accept that $\Sigma_N$ is non-negative definite, the equivalence of (i) and (ii) is evident.
  Therefore, it suffices to demonstrate only non-negative definiteness, which we prove by induction on $N$.

  When $N=1$, it is clearly true.
  Assuming it holds up to $N-1$, consider $\tilde{\Sigma}_N \in \mathbb{R}^{(N-1)\times(N-1)}$, which is the matrix $\Sigma_N$ with the $N$-th row and column removed.
  For $\boldsymbol{S}_N := (s_1, \ldots, s_N)^{\top} \in \mathbb{R}^N$, completing the square for the quadratic form in terms of $s_N$, we obtain
  \begin{align*}
    &\boldsymbol{S}_N^{\top}\Sigma_N\boldsymbol{S}_N\\
    =& T_N\left(1-\sum_{i=1}^{N-1} T_i\right)s_N^2
    + 2\sum_{i=1}^{N-1} T_iT_N s_i s_N 
    + \boldsymbol{S}_{N-1}^{\top}\tilde{\Sigma}_N\boldsymbol{S}_{N-1}\\
    =& T_N\left(1-\sum_{i=1}^{N-1} T_i\right)\left(s_N + \frac{\sum_{i=1}^{N-1} T_is_i}{1-\sum_{i=1}^{N-1} T_i}\right)^2
    + \boldsymbol{S}_{N-1}^{\top}\tilde{\Sigma}_N\boldsymbol{S}_{N-1}
    - \frac{T_N\left(\sum_{i=1}^{N-1} T_is_i\right)^2}{1-\sum_{i=1}^{N-1} T_i}\\
    =& T_N\left(1-\sum_{i=1}^{N-1} T_i\right)\left(s_N + \frac{\sum_{i=1}^{N-1} T_is_i}{1-\sum_{i=1}^{N-1} T_i}\right)^2
    + \frac{1-\sum_{i=1}^N T_i}{1-\sum_{i=1}^{N-1}T_i}\boldsymbol{S}_{N-1}^{\top}\Sigma_{N-1}\boldsymbol{S}_{N-1}\\
    \geq& 0.
  \end{align*}
  Thus, $\Sigma_{N}$ is shown to be non-negative definite.
\end{proof}

\begin{rem}
  In Example~\ref{ex_tuika}, the assumption of a binomial distribution Bin\((M,\theta)\) is not essential. 
  The argument holds for any probability distribution \(\tilde{p}(x|\theta)\) that satisfies Lemma~\ref{tuika_lem1} and (\ref{tuika_eq02}). 
  For example, a Poisson distribution Po\((\theta)\) with mean \(\theta\) can also easily be verified to meet these conditions.

\end{rem}

\section{Conclusion}
In this paper, we first elucidate the relationship between the Taylor expansion of the Kullback-Leibler divergence and the log-likelihood ratio function (Proposition~\ref{prop1}), and apply it to semi-regular models, i.e., models where the rank of the Fisher information matrix is non-zero. We provided formulas related to the Taylor expansion of the Kullback-Leibler divergence in Main Theorem~\ref{mainthm}. In Main Theorem~\ref{mainthm1}, we use the Taylor expansion derived in Main Theorem~\ref{mainthm} to perform a specific blow-up and obtain evaluations related to the real log canonical threshold. Particularly, we derived formulas that provide exact values for the real log canonical threshold under certain conditions of linear independence.

The above discussion requires Assumption~\ref{mainass}. In Section~\ref{sec:chg}, we provide a method for constructing variable transformations that satisfy Assumption~\ref{mainass}.

As specific examples using Main Theorem~\ref{mainthm1}, we presented formulas for the real log canonical threshold in cases where the parameter count is $d$ and the rank of the Fisher information matrix is $d$ or $d-1$ (Proposition~\ref{prop:deg2}), and provided the exact values for the learning coefficient of a mixture distribution with a constant mixing ratio (Example~\ref{ex_tuika}).

The real log canonical threshold can be calculated directly using Main Theorem~\ref{mainthm1} only when the realizable parameter set $\Theta_*$ consists of a single point. As future work, we intend to generalize the techniques used in this study and derive learning coefficients for cases where $\Theta_*$ is not a single point.

\section*{Acknowledgments}
I am grateful to Professor Joe Suzuki of Osaka University for teaching me the basics of Bayesian theory and for providing a research theme that bridges algebraic geometry and statistics. 
I also thank him for his valuable comments, advice, and for checking this manuscript. 
I also express my gratitude to Professor Sumio Watanabe, who proposed the concept of the learning coefficient.

\appendix  

\def\thesection{Appendix \Alph{section}}
\renewcommand{\theequation}{\Alph{section}.\arabic{equation}}

\section{Proof of propositions and lemmas in Section \ref{sec:mainthm}}
\begin{proof}[\textbf{Proof of Lemma~\ref{lemG}}]\leavevmode\par
  Consider elements of $S_{i_1,\ldots,i_{n+1}}$ that either include or exclude the sequence $U=(i_{n+1})$.
  For sequences that include $U=(i_{n+1})$, the differential of $\log{p(x|\theta)}$ corresponding to the sequence $((i_1,\ldots,i_n),(i_{n+1}))$ is
  \[
    \frac{\partial^n\log{p(x|\theta)}}{\partial\theta_{i_1}\cdots\partial\theta_{i_n}}\cdot\frac{\partial\log{p(x|\theta)}}{\partial\theta_{i_{n+1}}}
  \]
  and for others, it corresponds to
  \[
    G_{\theta_{i_1}\cdots\theta_{i_n}}(x,\theta)\cdot\frac{\partial\log{p(x|\theta)}}{\partial\theta_{i_{n+1}}}
  \]
  For example, the differential of $\log{p(x|\theta)}$ corresponding to the sequence $((i_1),$ $(i_2,\ldots,i_n),(i_{n+1}))$ is,
  \[
    \left(\frac{\partial\log{p(x|\theta)}}{\partial\theta_{i_1}}
    \frac{\partial^{n-1}\log{p(x|\theta)}}{\partial\theta_{i_2}\cdots\partial\theta_{i_n}}
    \right)
    \cdot
    \frac{\partial\log{p(x|\theta)}}{\partial\theta_{i_{n+1}}}
  \]
  where
  \[
    \frac{\partial\log{p(x|\theta)}}{\partial\theta_{i_1}}\frac{\partial^{n-1}\log{p(x|\theta)}}{\partial\theta_{i_2}\cdots\partial\theta_{i_n}}
  \]
  is one of the components constructing $G_{\theta_{i_1}\cdots\theta_{i_n}}(x,\theta)$.

  On the other hand, sequences that do not include $U=(i_{n+1})$ correspond to
  \[
    \frac{\partial  G_{\theta_{i_1}\cdots\theta_{i_n}}(x,\theta)}{\partial\theta_{i_{n+1}}}
  \]
  by the product rule. For instance, for the elements of $S_{i_1,\ldots,i_{n+1}}$ corresponding to the sequences $((i_1,i_{n+1}),(i_2,\ldots,i_n)),((i_1),(i_2,\ldots,i_{n+1}))$, the differential of $\log{p(x|\theta)}$ is,
  \begin{align*}
    &\frac{\partial^2\log{p(x|\theta)}}{\partial\theta_{i_1}\partial\theta_{i_{n+1}}}\frac{\partial^{n-1}\log{p(x|\theta)}}{\partial\theta_{i_2}\cdots\partial\theta_{i_n}}
      +\frac{\partial\log{p(x|\theta)}}{\partial\theta_{i_1}}\frac{\partial^{n}\log{p(x|\theta)}}{\partial\theta_{i_2}\cdots\partial\theta_{i_{n+1}}}\\
    =&\frac{\partial}{\partial\theta_{i_{n+1}}}\left(\frac{\partial\log{p(x|\theta)}}{\partial\theta_{i_1}}
      \frac{\partial^{n-1}\log{p(x|\theta)}}{\partial\theta_{i_2}\cdots\partial\theta_{i_n}}\right)
  \end{align*}
  Hence, the recurrence relation as stated in the lemma is verified.\\
\end{proof}

\begin{proof}[\textbf{Proof of Proposition~\ref{prop1}}]\leavevmode\par
\begin{itemize}
\item[(1)]
Demonstrate using induction on $n$. For $n=1$, the case is evident from
\[
\frac{\partial f(x|\theta)}{\partial \theta_{i_1}}=-\frac{\frac{\partial p(x|\theta)}{\partial \theta_{i_1}}}{p(x|\theta)}
\]
For a general $n$, using the induction hypothesis,
\begin{align*}
&\frac{\partial^{n+1}  f(x|\theta)}{\partial \theta_{i_1}\cdots \partial \theta_{i_{n+1}}}
=\frac{\partial}{\partial \theta_{i_{n+1}}}\left\{-\frac{\frac{\partial^{n} p(x|\theta)}{\partial \theta_{i_1}\cdots \theta_{i_n}}}{p(x|\theta)}+G_{\theta_{i_1}\cdots\theta_{i_n}}(x,\theta)\right\}\\
=&-\frac{\frac{\partial^{n+1} p(x|\theta)}{\partial \theta_{i_1}\cdots\partial \theta_{i_{n+1}}}}{p(x|\theta)} +\frac{\frac{\partial^{n} p(x|\theta)}{\partial \theta_{i_1}\cdots\partial \theta_{i_n}}}{p(x|\theta)}\cdot\frac{\frac{\partial p(x|\theta)}{\partial \theta_{i_{n+1}}}}{p(x|\theta)}
+ \frac{\partial G_{\theta_{i_1}\cdots\theta_{i_n}}(x,\theta)}{\partial\theta_{i_{n+1}}}\\
=&-\frac{\frac{\partial^{n+1} p(x|\theta)}{\partial \theta_{i_1}\cdots \partial \theta_{i_{n+1}}}}{p(x|\theta)} +\left\{G_{\theta_{i_1}\cdots\theta_{i_n}}(x,\theta)+\frac{\partial^n \log{p(x|\theta)}}{\partial\theta_{i_1}\cdots\partial\theta_{i_n}}\right\}\cdot \frac{\partial\log{p(x|\theta)}}{\partial\theta_{i_{n+1}}}\\
  &~~~+\frac{\partial G_{\theta_{i_1}\cdots\theta_{i_n}}(x,\theta)  }{\partial\theta_{i_{n+1}}}\\
=&-\frac{\frac{\partial^{n+1} p(x|\theta)}{\partial \theta_{i_1}\cdots \partial \theta_{i_{n+1}}}}{p(x|\theta)} + G_{\theta_{i_1}\cdots\theta_{i_{n+1}}}(x,\theta)
\end{align*}
The last equality uses Lemma~\ref{lemG}.

\item[(2)]
Given the assumption that differentiation and integration can be interchanged, using the result from (1), it suffices to show
\[
\left.\mathbb{E}_X\left[\frac{\frac{\partial^n p(X|\theta)}{\partial \theta_{i_1}\cdots\partial \theta_{i_n}}}{p(X|\theta)}\right]\right|_{\theta=\theta_*}=0
\] as
\begin{align*}
&\left.\mathbb{E}_X\left[\frac{\frac{\partial^n p(X|\theta)}{\partial \theta_{i_1}\cdots\partial \theta_{i_n}}}{p(X|\theta)}\right]\right|_{\theta=\theta_*}
= \int_{\chi} \frac{\left.\frac{\partial^n p(x|\theta)}{\partial \theta_{i_1}\cdots\partial \theta_{i_n}}\right|_{\theta=\theta_*}}{p(x|\theta_*)}\cdot q(x) dx\\
= &\int_{\chi} \frac{\left.\frac{\partial^n p(x|\theta)}{\partial \theta_{i_1}\cdots\partial \theta_{i_n}}\right|_{\theta=\theta_*}}{p(x|\theta_*)}\cdot p(x|\theta_*) dx
=\int_{\chi} \left.\frac{\partial^n p(x|\theta)}{\partial \theta_{i_1}\cdots\partial \theta_{i_n}}\right|_{\theta=\theta_*} dx \\
=&\left.\frac{\partial^n}{\partial \theta_{i_1}\cdots\partial \theta_{i_n}}\int_{\chi} p(x|\theta)dx\right|_{\theta=\theta_*}
=\left.\frac{\partial^n1}{\partial \theta_{i_1}\cdots\partial \theta_{i_n}}\right|_{\theta=\theta_*} =0
\end{align*}

\end{itemize}
\end{proof}


\begin{proof}[\textbf{Proof of Proposition~\ref{prop:taylor}}]\leavevmode\par
  \begin{itemize}
    \item [(1)]
      For any tuple of non-negative integers \((i_{r+1},\ldots,i_{d})\) that satisfies \(i_{r+1}+\cdots +i_{d}\leq 2m-1\), each term of \(G_{\theta_{r+1}^{i_{r+1}}\cdots\theta_d^{i_d}}\) takes the form:
      \begin{equation}\label{prop:taylor_eq01}
      \prod_{U\in T}\frac{\partial^{|U|}\log p(x|\theta)}{\prod_{k\in U}\partial\theta_k}
      \end{equation}
      where $T$ consists of multiple non-empty proper subsets of the set
      \[
        \{
          \underbrace{r+1, \ldots, r+1}_{\#=i_{r+1}}, \ldots, \underbrace{d, \ldots, d}_{\#=i_{d}}
        \}
      \]

      Thus, if \(\sum_{k=r+1}^d i_k\leq 2m-1\), it includes some \(U_1\) where \(|U_1|\leq m-1\).
      By assumption,
      \begin{equation}\label{prop:taylor_eq03}
      \left.\frac{\partial^{|U_1|}\log p(x|\theta)}{\prod_{k\in U_1}\partial\theta_k}
      \right|_{(\theta_1,\ldots,\theta_d)=0}= 0~~\text{a.s.}
      \end{equation}
      and hence from (\ref{prop:taylor_eq01}) it follows that \(G_{\theta_{r+1}^{i_{r+1}}\cdots\theta_d^{i_d}}(X,0)=0\) ~~\text{a.s.}.
    \item [(2)]
      Similarly to (1), each term of \(G_{\theta_j\theta_{r+1}^{i_{r+1}}\cdots\theta_d^{i_d}}\) takes the form of (\ref{prop:taylor_eq01}), where $T$ consists of multiple non-empty proper subsets of the set
      \[
        \{
        j,\underbrace{r+1,\ldots,r+1}_{\#=i_{r+1}},\ldots,\underbrace{d,\ldots,d}_{\#=i_{d}}
        \}
      \]
      Thus including some \(U_2\) that does not contain element \(\{j\}\).
      By assumption,
      \begin{equation}\label{prop:taylor_eq04}
      \left.\frac{\partial^{|U_2|}\log p(x|\theta)}{\prod_{k\in U_2}\partial\theta_k}
      \right|_{(\theta_1,\ldots,\theta_d)=0}= 0~~\text{a.s.}
      \end{equation}
      and hence from (\ref{prop:taylor_eq01}) it follows that \(G_{\theta_j\theta_{r+1}^{i_{r+1}}\cdots\theta_d^{i_d}}(X,0)=0\)~~\text{a.s.}.
    \item [(3)]
    As with (2), consider that all $T$ other than
    \[
    T_0 = (
    (j), (\underbrace{r+1, \ldots, r+1}_{\#=i_{r+1}}, \ldots, \underbrace{d, \ldots, d}_{\#=i_{d}})
    )
    \]
    contain terms in the form of (\ref{prop:taylor_eq04}) as factors and therefore do not need to be considered when computing $G$. Given that such $T_0$ is unique,
      \begin{align*}
      G_{\theta_j\theta_{r+1}^{i_{r+1}}\cdots\theta_{d}^{i_{d}}}(X,0)
      =&\prod_{U\in T_0}
      \left.\frac{\partial^{|U|}\log p(X|\theta)}{\prod_{k\in U}\partial\theta_k}\right|_{(\theta_1,\ldots,\theta_d)=0}\\
      =&\left.\frac{\partial \log p(X|\theta)}{\partial\theta_j}\right|_{(\theta_1,\ldots,\theta_d)=0}
      \left.\frac{\partial^{m} \log p(X|\theta)}{\partial\theta^{i_{r+1}}_{r+1}\cdots\partial\theta^{i_{d}}_d}\right|_{(\theta_1,\ldots,\theta_d)=0}~~\text{a.s.}\\
      =&\left.\frac{\partial f(X|\theta)}{\partial\theta_j}\right|_{(\theta_1,\ldots,\theta_d)=0}
      \left.\frac{\partial^{m} f(X|\theta)}{\partial\theta^{i_{r+1}}_{r+1}\cdots\partial\theta^{i_{d}}_d}\right|_{(\theta_1,\ldots,\theta_d)=0}
      \end{align*}
      is shown.
    \item [(4)]   
    Similar to (1),
    \begin{equation*}
      j_{r+1} + \cdots + j_d = m, \quad i_h \geq j_h \geq 0
    \end{equation*}
    for the integer tuples $(j_{r+1}, \ldots, j_d)$, define
    \[
    T_0 = (
    (\underbrace{r+1, \ldots, r+1}_{\#=j_{r+1}}, \ldots, \underbrace{d, \ldots, d}_{\#=j_d}),
    (\underbrace{r+1, \ldots, r+1}_{\#=i_{r+1}-j_{r+1}}, \ldots, \underbrace{d, \ldots, d}_{\#=i_d-j_d})
    )
    \]
    Since $T_0$ is the only tuple considered in calculating $G$ because other $T$s contain factors in the form of (\ref{prop:taylor_eq03}), note that the number of such $T_0$ depending on $(j_{r+1}, \ldots, j_d)$ is
    \[
    \binom{i_{r+1}}{j_{r+1}} \cdots \binom{i_d}{j_d} \times
    \begin{cases}
      \frac{1}{2} & \text{if } (2j_{r+1}, \ldots, 2j_d) = (i_{r+1}, \ldots, i_d)\\
      1 & \text{else}
    \end{cases}
    \]
    Given this, it has been demonstrated:
      \begin{align*}
      &G_{\theta_{r+1}^{i_{r+1}}\cdots\theta_d^{i_d}}(X,0)
      =\sum_{\substack{j_{r+1}+\cdots+j_d=m\\ i_h\ge j_h\ge0}}
      \prod_{U\in T_0}
      \left.\frac{\partial^{|U|}\log p(X|\theta)}{\prod_{k\in U}\partial\theta_k}\right|_{(\theta_1,\ldots,\theta_d)=0}\\
      =&\frac{1}{2}\sum_{\substack{j_{r+1}+\cdots+j_d=m\\ i_h\ge j_h\ge0}}
        \binom{i_{r+1}}{j_{r+1}}\cdots\binom{i_d}{j_d}
        \left.\frac{\partial^m \log p(X|\theta)}{\partial\theta_{r+1}^{j_{r+1}}\cdots\partial\theta_d^{j_d}}\right|_{\theta=0}\\
        &~~~\times
        \left.\frac{\partial^m \log p(X|\theta)}{\partial\theta_{r+1}^{i_{r+1}-j_{r+1}}\cdots\partial\theta_d^{i_d-j_d}}\right|_{\theta=0} ~~\text{a.s.}\\
      =&\frac{1}{2}\sum_{\substack{j_{r+1}+\cdots+j_d=m\\ i_h\ge j_h\ge0}}
        \binom{i_{r+1}}{j_{r+1}}\cdots\binom{i_d}{j_d}
        \left.\frac{\partial^m f(X|\theta)}{\partial\theta_{r+1}^{j_{r+1}}\cdots\partial\theta_d^{j_d}}\right|_{\theta=0}\\
        &~~~\times
        \left.\frac{\partial^m f(X|\theta)}{\partial\theta_{r+1}^{i_{r+1}-j_{r+1}}\cdots\partial\theta_d^{i_d-j_d}}\right|_{\theta=0}
      \end{align*}
      Note that in the transformation, we used the fact that
      \[
        \binom{i_{r+1}}{i_{r+1}-j_{r+1}}\cdots\binom{i_d}{i_d-j_d}=\binom{i_{r+1}}{j_{r+1}}\cdots\binom{i_d}{j_d}
      \]
  \end{itemize}
\end{proof}


\begin{proof}[\textbf{Proof of Lemma~\ref{mainlem2}}]\leavevmode\par
  \begin{itemize}
    \item [(1)]
      From Assumption~\ref{mainass}(1), the $r$ random variables
      \begin{equation}\label{eq_mainlem2}
        \left.\frac{\partial f(X|\theta)}{\partial\theta_1}\right|_{(\theta_1,\ldots,\theta_d)=0}
        ,\ldots,
        \left.\frac{\partial f(X|\theta)}{\partial\theta_r}\right|_{(\theta_1,\ldots,\theta_d)=0}
      \end{equation}
      are linearly independent, therefore,
      \begin{align*}
        \mathbb{E}_X\left[F^2_1(X|\theta_{1},\ldots,\theta_{r})\right]=0 
        &\Leftrightarrow  F_1(X|\theta_{1},\ldots,\theta_{r})=0 ~~\text{a.s.}\\
        &\Leftrightarrow  \sum_{k=1}^r\theta_k\cdot\left.\frac{\partial f(X|\theta)}{\partial\theta_k}\right|_{(\theta_1,\ldots,\theta_d)=0}=0 ~~\text{a.s.}\\
        &\Leftrightarrow  (\theta_1,\ldots,\theta_d)=0
      \end{align*}
      follows.

    \item[(2)]
      If Assumption~\ref{mainass}(3)(i) holds, then $F_m(X|\theta_{r+1},\ldots,\theta_{d})=0$ (a.s.) and the lemma holds as per (i).
      Therefore, it suffices to consider only when Assumption~\ref{mainass}(3)(ii) holds.
      \begin{align*}
        &\mathbb{E}_X\left[\left\{F_1(X|\theta_{1},\ldots,\theta_{r})+a\cdot F_m(X|\theta_{r+1},\ldots,\theta_{d})\right\}^2\right]=0\\
        \Leftrightarrow ~&F_1(X|\theta_{1},\ldots,\theta_{r})+a\cdot F_m(X|\theta_{r+1},\ldots,\theta_{d})=0~~\text{a.s.}\\
        \Leftrightarrow ~&\sum_{k=1}^r\theta_k\cdot\left.\frac{\partial f(X|\theta)}{\partial\theta_k}\right|_{(\theta_1,\ldots,\theta_d)=0}+a\cdot F_m(X|\theta_{r+1},\ldots,\theta_{d})=0~~\text{a.s.}
      \end{align*}
      Noting that $a>0$ and by Assumption~\ref{mainass}(3)(ii), since (\ref{eq_mainlem2}) and the random variable $a\cdot F_m(X|\theta_{r+1},\ldots,\theta_{d})$ are linearly independent, the equivalence does not hold, and we obtain
      \[
      \mathbb{E}_X\left[\left\{F_1(X|\theta_{1},\ldots,\theta_{r})+a\cdot F_m(X|\theta_{r+1},\ldots,\theta_{d})\right\}^2\right]>0
      \]
  \end{itemize}
\end{proof}

\section{Mathematica Output Results}



\end{document}